\documentclass[12pt,a4paper]{amsart}

\usepackage{amsmath,amssymb,amsthm,enumerate,graphicx}
\usepackage[all]{xy}
\usepackage{calligra,mathrsfs}
\usepackage{hyperref}
\usepackage{bm,color}
\usepackage{amsfonts}
\usepackage{amscd}
\usepackage[latin2]{inputenc}
\usepackage{t1enc}
\usepackage[mathscr]{eucal}
\usepackage{indentfirst}
\usepackage{graphics}
\usepackage{pict2e}
\usepackage{epic}
\numberwithin{equation}{section}
\usepackage[margin=2.9cm]{geometry}
\usepackage{epstopdf} 
\usepackage{tikz-cd}
\usepackage{mathtools}
\usepackage{cases}

\theoremstyle{plain}
\newtheorem{thm}{Theorem}[section]

\theoremstyle{ams-restatedtheorem}
\newtheorem*{restatedtheorem*}{}
{\end{restatedtheorem*}}
\makeatother

\theoremstyle{question}
\newtheorem{question}[thm]{Question}

\theoremstyle{definition}
\newtheorem{defn}[thm]{Definition}

\theoremstyle{proposition}
\newtheorem{prop}[thm]{Proposition}

\theoremstyle{remark}
\newtheorem{rem}[thm]{Remark}

\theoremstyle{remark}

\theoremstyle{corollary}
\newtheorem{cor}[thm]{Corollary}

\theoremstyle{lemma}
\newtheorem{lem}[thm]{Lemma}

\theoremstyle{conjecture}
\newtheorem{cnj}[thm]{Conjecture}

\theoremstyle{question}

\newcommand{\widesim}[2][1.5]{
  \mathrel{\overset{#2}{\scalebox{#1}[1]{$\sim$}}}
}

\newcommand{\bigcdot}{\boldsymbol{\cdot}}
\newcommand{\Hom}{\text{Hom}}
\newcommand{\bs}{\backslash}
\newcommand{\mf}{\mathfrak}
\newcommand{\mc}{\mathcal}
\newcommand{\grade}{\text{grade}}

\DeclareMathOperator{\Homs}{\mathscr{H}\text{\kern -3pt {\calligra\large om}}\,}
\DeclareSymbolFont{extraup}{U}{zavm}{m}{n}
\DeclareMathSymbol{\varheart}{\mathalpha}{extraup}{86}
\DeclareMathSymbol{\vardiamond}{\mathalpha}{extraup}{87}
\DeclareMathSymbol{\varspade}{\mathalpha}{extraup}{81}

     \title{Homogeneous Liaison and the Sequentially Bounded Licci Property}

     \author{Jesse Keyton}
     \address{University of Arkansas, USA}
     \email{jskeyton@uark.edu}

     \date{August 1, 2019}

     \subjclass{commutative algebra}

\begin{document}
     \begin{abstract}
     In CI-Liaison, significant effort has been made to study ideals that are in the linkage class of a complete intersection, which are called licci ideals. In a polynomial ring, recently E. Chong defined a "sequentially bounded" condition on the degrees of the forms generating the regular sequences, and used this condition to find a large class of licci ideals satisfying the Eisenbud-Green-Harris Conjecture (among them, grade $3$ homogeneous Gorenstein ideals). He raised the question of whether all homogeneous licci ideals are sequentially bounded licci. In this paper we construct a class of examples that are homogeneously licci, but not sequentially bounded licci, thus answering his question in the negative. The structure of certain Betti tables plays a central role in our proof.
     \end{abstract}
     \maketitle

\section{Introduction} 

Let $R$ be a Gorenstein ring, and let $I$ and $J$ be proper $R$-ideals. Then $I$ and $J$ are \textit{(CI-) directly linked}, written as $I\sim J$, if there is a regular sequence $\mathbf{f} = f_1,\dots, f_g$ for which $I=(\mathbf{f}):J$ and $J=(\mathbf{f}):I$. The \textit{linkage class} (or \text{liaison class}) of an ideal $J$ consists of all ideals $I$ for which there is a sequence of direct links
\[ I = I_0 \sim I_1 \sim \dots \sim I_n = J.\]
Ideals in the linkage class of a complete intersection are called \textit{licci} ideals, which are of special interest, in large part due to their nice homological properties (see \cite{B81}, \cite{H80}, \cite{HK82}, \cite{HU85}, \cite{HU89}, \cite{HU92}, \cite{J98}, \cite{J01}, \cite{U87}, \cite{U89} for a sample of results). 

\par Given a licci ideal, there are possibly many different sequences of links to a complete intersection. In \cite{HU88}, Huneke and Ulrich provide an optimal sequence when $R$ is local with an infinite residue field; namely, fixing the licci ideal, it is possible to construct at each step a regular sequence from a minimal generating set to create the next link. However, applying their result in the homogeneous setting causes the graded structure of the ideal to be lost. In fact, for a homogeneous licci ideal in a polynomial ring, it is not known if it is possible to use only homogeneous regular sequences to create links to a complete intersection. In the case where it is possible, we use the term \textit{homogeneously licci}. Several papers investigate the question of if there is a numerical optimal (i.e., optimal constraints on the degrees of the forms generating the regular sequences) way to link a homogeneously licci ideal to a complete intersection. 
\par  When $R$ is local, grade $2$ perfect ideals and grade $3$ Gorenstein ideals are licci (by \cite{G54} and \cite{W73}, respectively). The work of Gaeta in \cite{G54} shows that grade $2$ homogeneous perfect ideals are \textit{minimally licci}, which means it is possible to link them to a complete intersection by using regular sequences with forms of smallest possible degree at each step. In \cite{MN08}, Migliore and Nagel showed grade $3$ homogeneous Gorenstein ideals are minimally licci as well. However, in the the main result of \cite{HMNU07}, Huneke, Migliore, Nagel, and Ulrich construct a class of ideals that are homogeneously licci but not minimally licci. Subsequently, E. Chong introduced a weaker \textit{sequentially bounded licci (SBL)} property which, roughly speaking, relaxes but still maintains control of the degrees of the forms (see Defintion \ref{SBL}.ii. for precise statement). We show that all of the non-minimally licci ideals constructed in \cite{HMNU07}, and minimal links of these ideals, are SBL (proved in Propositions \ref{t2 l1} and \ref{Koszul in L}). Further, all zero-dimensional licci monomial ideals are SBL by a result of Huneke and Ulrich in \cite{HU07}. Thus, it is natural to ask if all homogeneously licci ideals are SBL. 

\par Chong asks a slightly more general question about the following characterizations of licci ideals:
\[ \fbox{ $\text{minimally licci} \Rightarrow \text{sequentially bounded licci} \Rightarrow \text{homogeneously licci} \Rightarrow \text{licci}.$ }\]

\begin{question}[\cite{Chong15}, Question 10.3.6] \label{SBL Question}
Knowing there are homogeneously licci ideals that are not minimally licci (by \cite{HMNU07}), which of the characterizations above are not equivalent?
\end{question}

The following illustrates why one is interested to know whether all homogeneously licci ideals are SBL: using the SBL defintion and a Hilbert function formula for a direct homogeneous link (see Propositon \ref{HF links}), Chong provided further evidence for the Eisenbud-Green-Harris conjecture. Before stating the conjecture and Chong's result, we recall the following definition: if $C=(f_1,\dots, f_n)$ is a complete intersection of grade $n$ with $\deg f_1 \leq \dots \leq \deg f_n$, then $\text{type}(C)=(\deg f_1,\dots, \deg f_n)$.


\begin{cnj}[EGH, \cite{EGH96}]
If a homogeneous ideal $I$ in $R=k[x_1,\dots, x_N]$ contains a complete intersection of type $(a_1,\dots, a_n)$, then there is a monomial ideal containing $(x_1^{a_1},\dots, x_n^{a_n})$ with the same Hilbert function as $I$. 
\end{cnj}

\begin{thm}[\cite{Chong16}] \label{Chong EGH}
Suppose $I$ is a grade $n$ homogeneous licci ideal of $R=k[x_1,\dots, x_N]$. Further, suppose $I_s$ is a complete intersection and that there is a sequence of homogeneous links
\[ I = I_0 \widesim{C_0} \dots \widesim{C_{s-1}} I_s,\]
such that (i.) $\text{type}(C_{j}) \geq \text{type}(C_{j+1})$ for $j=0,\dots, s-1$, and (ii.) $\text{type}(C_0)$ is minimal in $I$. Then $I$ satisfies the EGH Conjecture.
\end{thm}

Condition (i.) means that if $\text{type}(C_j) = (a_1,\dots, a_n)$ and $\text{type}(C_{j+1}) = (b_1,\dots, b_n)$, then $a_i\geq b_i$ for $i=1,\dots,n$. This non-increasing sequence of types is what Chong defines to be the \textit{sequentially bounded} property. The minimality in condition (ii.) is well-defined by this order on the types. Even when (ii.) does not hold, a similar idea shows that SBL ideals provide evidence for a positive answer to the following question asked by Huneke and Ulrich.

\begin{question}[\cite{HU07}] \label{Monomial HF Q}
If $I$ is a zero-dimensional homogeneously licci ideal, is there a licci monomial ideal with the same Hilbert Function as $I$?
\end{question}

\par Our main result is twofold: (1) we construct a class of ideals that are homogeneously licci but not sequentially bounded licci, further answering Question \ref{SBL Question}; and (2) restricting to three variables, for each ideal in this class we construct a zero-dimensional licci monomial ideal with the same Hilbert function, providing further evidence for a positive answer to Question \ref{Monomial HF Q}. A plentiful source of examples for (1) come from direct links of certain non-minimally licci ideals of \cite{HMNU07}, constructed with a carefully chosen regular sequence (see Section 3.1).
\par Proving directly that an ideal does not have the SBL property is difficult, since there are infinitely many complete intersection types in an ideal. Even if two complete intersections have the same type, the relations between the generators may provide different properties of the link. Further, after choosing the complete intersection to define the first link, another complication arises when repeating complete intersection types. Indeed, using the same complete intersection type for any number of links does not violate the sequentially bounded condition.  
\par We overcome these difficulties by a deep study of certain Betti tables and the Betti tables induced by homogeneous links. In particular, we find a class of ideals whose graded Betti numbers, and the graded Betti numbers of direct links, help determine the minimal complete intersection type of the ideal. 
\par The following theorems summarize the results of this paper, the first being our main contribution to the study of homogeneous linkage.

\newtheorem*{thm:main}{Theorem \ref{thm:main} (Main Theorem)} 
\begin{thm:main} 
Fix $n\geq 4$. For a polynomial ring $R=k[x_1,\dots, x_N]$, with $k$ a field and $N\geq 3$, assume $J$ is a grade $3$ homogeneous $R$-ideal with a minimal graded free resolution 
\[ 0 \to \begin{matrix} R(-(2n+4)) \\ \oplus \\ R(-(2n+7)) \\ \oplus \\ R(-(3n+5)) \end{matrix} \to \begin{matrix} R(-(n+2))^2 \\ \oplus \\ R(-(n+4)) \\ \oplus \\ R(-(2n+3)) \\ \oplus \\ R(-(2n+6))^2 \\ \oplus \\ R(-(3n+4)) \end{matrix} \to \begin{matrix} R(-2) \\ \oplus \\ R(-(n+1))^2 \\ \oplus \\ R(-(n+3)) \\ \oplus \\ R(-(2n+4)) \end{matrix} \to J \to 0.\]
Then $J$ is homogeneously licci, but not sequentially bounded licci. 
\end{thm:main}

\newtheorem*{thm:main2}{Theorem \ref{thm:main2}}
\begin{thm:main2}
When $N=3$, for each ideal satisfying the conditions of Theorem \ref{thm:main}, there is a zero-dimensional monomial licci ideal with the same Hilbert function.
\end{thm:main2}

\par The latter theorem shows that while the SBL property may be detected by the Betti table, it cannot be detected by the Hilbert function. Our proof of this part uses an interesting connection between Macaulay's inverse system and linkage. 

\par We remark here that in the geometric context (Gorenstein liaison) there is a property similar to the sequentially bounded property called "descending biliaison," which was studied as an attempt to classify the ideals in the Gorenstein linkage class of a complete intersection (see \cite{H02} and \cite{HSS06}). 
\par The outline of this paper is as follows. In Section 2 we provide background material, focusing on Ferrand's mapping cone resolution for direct and double homogeneous links in grade 3. Section 3 recalls the non-minimally licci ideals from \cite{HMNU07} and introduces several new classes of licci ideals, one of which is our class of non-SBL ideals. Section 4 contains the final propositions needed for the proof of Theorem \ref{thm:main}, all of which study different cases of double links of our proposed non-SBL ideals. In Section 5, we briefly recall Macaulay's inverse systems and prove Theorem \ref{thm:main2}. Lastly, there are two appendices: Appendix A gathers many remarks and lemmas that are used often to prove the statements in Sections 3 and 4, and Appendix B contains technical arguments about when ghost terms of double links trim. 
\par For further background in linkage, see \cite{HU87}.
\par \textit{Acknowledgement}. The author is grateful to his Ph.D. advisor, Paolo Mantero, for his many helpful conversations. Without him this paper would not be possible.


\section{Ferrand's Mapping Cone and Minimal Koszul Relations}

In this section we review background material in homogeneous linkage, with a focus on Ferrand's mapping cone construction. In particular, we give an explicit description for when there are trims in the induced resolution of a direct link, which reveals the significance of minimal Koszul relations in determining the licci property for grade $3$ ideals. Unless otherwise stated, the ring $R$ is the polynomial ring $R=k[x_1,\dots, x_N]$ over a field $k$ and $\mf{m}$ is the homogeneous maximal ideal. Further, we assume $R$ has the standard grading.

\begin{defn}
Assume $C=(f_1, \dots, f_n)\subseteq R$ is a grade $n$ homogeneous complete intersection with $\deg f_1\leq \dots \leq \deg f_n$. 
\begin{enumerate}[i.]
\item The \textit{type} of $C$ is $\text{type}(C) = (\deg f_1,\dots, \deg f_n)$. 
\item If $\text{type}(D)=(b_1, \dots, b_n)$, then $\text{type}(C)\leq \text{type}(D)$ if and only if $\deg f_i \leq b_i$ for all $i=1,\dots, n$. 
\item For a grade $n$ ideal $I$, this order yields a unique minimal element in the set $\{ \text{type}(D): D\subseteq I \text{ is a complete intersection of grade } n\}$, which we call the \textit{minimal type} of $I$. 
\end{enumerate}
\end{defn}

\begin{defn} \label{grade jump def}
Let $I$ be a grade $n$ perfect $R$-ideal with $C\subseteq I$ a complete intersection of minimal type $(a_1,\dots, a_n)$.
\begin{enumerate}[i.]
\item The number $g_j(I) := a_j$ is called the \textit{grade} $j$ \textit{jump} of $I$, for $j=1,\dots, n$.
\item The direct link $C:I$ is called a \textit{minimal link} of $I$.
\end{enumerate}
\end{defn}

Using the definition in the introduction, it may not be clear how to determine when two ideals are directly linked. However, the following proposition demonstrates the typical method used to construct links. 

\begin{prop}[\cite{PS74}]
Let $R$ be a local Gorenstein ring, and let $I$ be a proper, unmixed $R$-ideal. If $C$ is a complete intersection with $\grade(C)=\grade(I)$, then $I$ and $C:I$ are directly linked by $C$. 
\end{prop}

\begin{defn}
Let $I$ be a grade $n$ homogeneous perfect $R$-ideal with a minimal graded free resolution
\begin{equation} \label{gen res}
\mc{F}_{\bigcdot}: 0 \to \oplus_j R(-j)^{\beta_{n,j}} \to \oplus_j R(-j)^{\beta_{n-1,j}} \to \dots \to \oplus_j R(-j)^{\beta_{1,j}} \to R \to 0. 
\end{equation}
\begin{enumerate}[i.]
    \item The graded Betti numbers of $R/I$ are the numbers $\beta_{i,j}$, denoted $\beta_{i,j}(R/I)$.
    \item The Betti table of $R/I$ is the table whose columns, $i$, correspond to the homological degrees, and whose rows, $j$, correspond to the degrees of the generators of the syzygy modules. The entry in position $i,j$ is $\beta_{i,i+j}(R/I)$.
    \item For a fixed $i$, the total Betti number $\beta_i(R/I)$ is the rank of $F_i$, the free module in the $i$th step of $\mc{F}_{\bigcdot}$.
    \item The \textit{deviation} of $I$ is the value $d(I):= \beta_1(R/I)-n$.
    \item The \textit{CM type} of $I$ is the value $r(R/I):=\beta_n(R/I)$.
\end{enumerate}
\end{defn}

\begin{thm}[Ferrand's Mapping Cone, \cite{PS74}] \label{Ferrand}
Let $I$ be a grade $n$ homogeneous perfect $R$-ideal with a minimal graded free resolution as (\ref{gen res}). Suppose $C\subsetneq I$ is a complete intersection of type $(a_1,\dots, a_n)$, set $a=a_1+\dots +a_n$, and let $\mc{K}_{\bigcdot}$ be the Koszul resolution of $C$. Further, let $u_{\bigcdot}:\mc{K}_{\bigcdot} \to \mc{F}_{\bigcdot}$ be a map of resolutions with $u_0 = id_R$. If $\mc{G}_{\bigcdot}(u_{\bigcdot})$ is the mapping cone complex induced by $u_{\bigcdot}$, then the reduced shifted dual complex $\mc{G}_{\bigcdot}(u_{\bigcdot})^*[-n-1]$, which can be written as
\[ 0 \to \oplus_{j} R(-(a-j))^{\beta_{1,j}} \to \begin{matrix} \oplus_{i=1}^{n} R(-(a-a_i)) \\ \oplus \\ \oplus_j R(-(a-j))^{\beta_{2,j}} \end{matrix} \dots \to \begin{matrix} \oplus_{i=1}^{n} R(-a_i) \\ \oplus \\ \oplus_j R(-(a-j))^{\beta_{n,j}} \end{matrix} \to R \to 0, \]
is a graded free resolution of $R/(C:I)$.
\end{thm}

\begin{rem}
The resolution of $R/(C:I)$ is \textit{reduced}, meaning that since the isomorphism $u_0=id_R:R\to R$ does not contribute to the minimal graded free resolution, it was \textit{trimmed} off. Notice also that this resolution of $R/(C:I)$ may not be minimal. 
\end{rem}


We now record the details of this construction for our setting, but the general construction can be found in \cite{W94}. Let $I$ be a grade $3$ perfect $R$-ideal containing a grade $3$ complete intersection $C=(f_1,f_2,f_3)$. Assume $I\neq C$, and set $J=C:I$. For an $R$-module $M$, set $M^*=\Hom_R(M,R)$. Now let $(\mc{F}_{\bigcdot},d_{\bigcdot})$ be the minimal free resolution of $I$, and let $(\mc{K}_{\bigcdot},\delta_{\bigcdot})$ be the Koszul resolution of $C$, with differential maps
\[ \delta_1=[ f_1 \hspace{3pt} f_2 \hspace{3pt} f_3 ], \hspace{7pt} \delta_2 = \left[ \begin{matrix} -f_2 & f_3 & 0 \\ f_1 & 0 & -f_3 \\ 0 & -f_1 & f_2 \end{matrix} \right], \hspace{3pt} \text{and} \hspace{7pt} \delta_3 = \left[ \begin{matrix} f_3 \\ f_2 \\ f_1 \end{matrix} \right] .\]

\noindent Define $u_{\bigcdot}: (\mc{F}_{\bigcdot},d_{\bigcdot}) \to (\mc{K}_{\bigcdot}, \delta_{\bigcdot})$ to be the map of resolutions extending the inclusion $C\subseteq I$, and let $(\mc{G}^*_{\bigcdot},\partial_{\bigcdot})$ be the dual mapping cone of $u_{\bigcdot}$. Then $\mc{G}^*_{\bigcdot}$ is a resolution of $J$, and can be written as 
\[ 0\to F_1^* \xrightarrow{\partial_3} \begin{matrix} K_1^* \\ \oplus \\ F_2^* \end{matrix} \xrightarrow{\partial_2} \begin{matrix} K_2^* \\ \oplus \\ F_3^* \end{matrix} \xrightarrow{\partial_1} R \to R/J \to 0, \]
where
\[ \partial_1 = [ -\delta_3^{*} \hspace{5pt} u_3^{*} ], \hspace{7pt} \partial_2 = \left[ \begin{matrix} -\delta_2^{*} & u_2^* \\ 0 & d_3^* \end{matrix} \right], \hspace{3pt} \text{and} \hspace{7pt} \partial_3 = \left[ \begin{matrix} u_1^* \\ d_2^* \end{matrix} \right]. \]
Notice that the first three columns of $\partial_2$ represent Koszul relations. 
\par Since $\mc{F}_{\bigcdot}$ and $\mc{K}_{\bigcdot}$ are minimal, there are no unit entries in $d_i$ or $\delta_i$ for any $i$. Thus all unit entries of $\partial_i$ come from unit entries in the connecting maps $u_i$. We examine what these unit entries say about the structure of either the original ideal $I$ or the complete intersection $C$. First, a defintion.

\begin{defn}
Let $I$ be an $R$-ideal containing distinct elements $f,g_1,\dots, g_t$. We say that $f$ and $g_1,\dots, g_t$ are \textit{weak associates in} $I$ if there exists elements $u_1,\dots, u_t\in R \bs \mf{m}$ and $h\in \mf{m}I$ for which $f=u_1g_1+\dots +u_tg_t +h$. When $I$ is understood, we simply say that $f$ and $g_1,\dots, g_t$ are \textit{weak associates}.
\end{defn}

\noindent \textit{Unit entries in} $u_1$: Fix a minimal generating set of $I$ and let $\{e_1,e_2,e_3\}$ be a basis for $K_1$. If $u_1(e_i) \in F_1$ has a unit entry, then $f_i$ and some minimal generators of $I$ are weak associates. Thus unit entries of $u_1$ correspond to using part of a minimal generating set of $I$ for generators of $C$.

\noindent \textit{Unit entries in} $u_2$: Let $\{e_1\wedge e_2, e_1\wedge e_3, e_2\wedge e_3\}$ be the ordered basis of $K_2$ that agrees with the differentials above. The mapping cone figure illustrates that there is a unit entry in $u_2$ if and only if $f_i\in \mf{m}J$ for some $i$ (since $u_3$ cannot have unit entries). We may assume $i=1$. This together with the definition of $\partial_2$ shows that $f_1\in \mf{m}J$ if and only if $u_2(e_2\wedge e_3)$ has a unit entry, which is equivalent to $u_2(e_2\wedge e_3) \in \text{Im}(d_2) / \mf{m} \text{Im}(d_2)$. By commutativity of $u_{\bigcdot}$, this is equivalent to the Koszul relation between $f_2$ and $f_3$ representing a minimal first syzygy of $I$. Notice that in this case, $\{f_2,f_3\}$ must be part of a minimal generating set of $I$. 

\noindent \textit{Unit entries in} $u_3$: If $u_3$ had a unit entry, then $J=R$, which would imply $C=I$. However, we started with the assumption that $C\neq I$, so there cannot be unit entries in $u_3$.


\subsection{Minimal Koszul Relations}
For the remainder of this section, we discuss in further detail the case where there are unit entries in $u_2$ from the set-up above; that is, the case where there are \textit{minimal Koszul relations}.

\begin{defn} \label{syzygy defs}
Let $I$ be a perfect homogeneous $R$-ideal with a minimal graded free resolution $(\mc{F}_{\bigcdot},d_{\bigcdot})$.
\begin{enumerate}[i.]
\item Let $\text{Syz}_1(R/I)=\ker d_1$ denote the module of first syzygies of $R/I$. If $\{e_1,\dots, e_t\}$ is a basis of $F_1$ with $d_1(e_i)=f_i$, and $s =g_1e_1+\dots +g_te_t$ is a homogeneous element of $\text{Syz}_1(R/I)$, then the degree of $s$ is $\deg s = \deg g_i + \deg f_i$, for any nonzero $g_i$.
\item The $R$-module $\text{Syz}_1(R/I)_{\leq d}$ consists of all first syzygies of degree $d$ or less. 
\item If $f,g\in I$ and the Koszul relation between $f$ and $g$ is a minimal syzygy of $I$, then we say that $f$ and $g$ yield a \textit{minimal Koszul relation} in $I$.
\item Set $C=\text{Im}(d_2)$ and let $K$ be the submodule of $C$ generated by the Koszul relations on the entries of $d_1$. The number of distinct minimal Koszul relations in $I$ is the value $$\lambda(I) := \dim_k (K+\mf{m}C)/\mf{m} C.$$
\end{enumerate}
\end{defn}

The value $\lambda(-)$ was used by Kustin and Miller in \cite{KM81} and \cite{KM82} to distinguish classes of grade $4$ Gorenstein ideals, and was used by Anne Brown in \cite{B87} to classify all grade 3, CM type 2 ideals with $\lambda = 1$ in Noetherian local rings. 
\par The work above proves the following proposition.

\begin{prop} \label{min Koszul relation}
Suppose $I$ is a grade $3$ perfect $R$-ideal containing a grade $3$ complete intersection $C$ with $C\neq I$. Let $(\mc{G}^*_{\bigcdot}, \partial_{\bigcdot})$ be Ferrand's reduced mapping cone resolution of $J=C:I$. Then 
\begin{enumerate}[i.]
\item There is a unit entry in $\partial_3$ if and only if a subset of a minimal generating set of $C$ is also part of a minimal generating set of $I$. In this case, the number of rows with a unit entry is the size of the largest such subset of $C$; and
\item There is a unit entry in $\partial_2$ if and only if some pair of minimal generators of $C$ yield a minimal Koszul relation in $I$. In this case, the number of rows with a unit entry is the number of such pairs of generators for a fixed generating set of $C$.
\end{enumerate}
\end{prop}

\begin{cor}[\cite{B87}, Lemma 2.4] \label{trims}
With the same set-up as Proposition \ref{min Koszul relation}, the minimal free resolution of $I$ may be obtained by trimming the resolution $\mc{G}^*_{\bigcdot}$, and the trims are determined by the unit entries in the maps $\partial_2$ and $\partial_3$.
\end{cor}

\begin{cor} \label{construct Koszul}
Let $I$ be a grade $3$ perfect $R$-ideal containing the grade $3$ complete intersection $C=(f_1,f_2,f_3)$. If $f_1\in \mf{m}I$, then $f_2$ and $f_3$ yield a minimal Koszul relation in $C:I$.
\end{cor}

Our approach to prove the main theorem is to study double links of the proposed non-SBL ideals. The previous statements about minimal Koszul relations play an important role in understanding these double links.    
\begin{defn}
Assume $I$ is a homogeneous perfect $R$-ideal, and $C$ and $D$ are homogeneous complete intersections such that
$$I\widesim{C} I_1\widesim{D} I_2.$$ 
If $\text{type}(C)\geq \text{type}(D)$, then way say that $I_2$ is a \textit{sequentially bounded double link} of $I$ with respect to $\text{type}(C)$.
\end{defn}

The classes of ideals we consider in the next section each have $\lambda=1$. Avoiding the minimal Koszul relation, double links of these ideals with repeated types has little effect on the minimal graded free resolution. We formalize this with a definition.

\begin{defn} \label{ghost def}
Let $I$ be a grade $3$ homogeneous perfect $R$-ideal containing a complete intersection $C$ of type $(a_1,a_2,a_3)$. Consider 
\[ I \widesim{C} I_1 \widesim{D} I_2,\]
where $\text{type}(D)=\text{type}(C)$. If $d(I_1) = r(R/I)$ (equivalently, $C$ avoids the minimal Koszul relations in $I$) and the minimal type of $I_1$ is $(a_1,a_2,a_3)$, then we say $I_2$ is a \textit{ghost double link} of $I$ with respect to $(a_1,a_2,a_3)$.  
\end{defn}

The word ghost is motivated by this observation: if $I_2$ is a ghost double link of $I$ with respect to $(a_1,a_2,a_3)$, and if $\mc{F}_{\bigcdot}$ is a minimal free resolution of $I$, then applying Ferrand's mapping cone twice shows that $\mc{F}_{\bigcdot}'$ is a free resolution of $I_2$, where 
\begin{align} \notag
F_1' &= F_1\oplus R(-a_1) \oplus R(-a_2) \oplus R(-a_3), \\ \notag
F_2' &= F_1\oplus R(-a_1) \oplus R(-a_2) \oplus R(-a_3), \text{ and} \\ \notag
F_3' &= F_3. \notag
\end{align}
The terms $R(-a_i)$ may or may not trim, and thus are examples of \textit{ghost terms}. In particular, for $1\leq j\leq 3$, we have $$\beta_{i,a_j}(R/I)\leq \beta_{i,a_j}(R/I_2)\leq \beta_{i,a_j}(R/I)+1$$
for $1\leq i\leq 2$, and $\beta_{3,j}(R/I)=\beta_{3,j}(R/I_2)$ for all $j$. There are, however, common scenarios that guarantee at least some of the ghost terms do trim.  

\begin{lem} \label{ci shift trims}
Assume $I_2$ is a ghost double link of $I$ with respect to $(a_1,a_2,a_3)$. If $C$ and $D$ are the intermediate complete intersections as above, then write $C=(g_1,g_2,g_3)$ and $D=(f_1,f_2,f_3)$ where $a_i=\deg f_i = \deg g_i$. Then for a fixed $j$, we have $\beta_{1,a_j}(R/I_2)=\beta_{1,a_j}(R/I)$ and $\beta_{2,a_j}(R/I_2) = \beta_{2,a_j}(R/I)$ if either of the following hold:
\begin{enumerate}[i.]
\item $g_j$ a minimal generator of $I$.
\item $g_j \in \mf{m}I$, and $f_k, f_{\ell}$ are weak associates of $g_k,g_{\ell}$ in $I_1$, respectively, where $j,k,$ and $\ell$ are distinct.
\end{enumerate}
\end{lem}

\begin{proof}
Assume $j=1$. If $f_1$ is a minimal generator of $I$, then by Proposition \ref{min Koszul relation}.i., the term $R(-(a_2+a_3))$ in steps 2 and 3 of the resolution of $I_1$ can trim. In Ferrand's mapping cone resolution of $I_2=D:I_1$, these terms take the form $R(-a_1)$ in steps 1 and 2. Since they can trim, the only possible difference in the Betti tables of $I$ and $I_2$ is from the ghost terms $R(-a_2)$ and $R(-a_3)$ in steps 1 and 2 of the resolution of $I_2$.
\par To prove ii., let $K=\text{Syz}_1(R/I_1)$. If $g_j\in \mf{m}I$, then by Corollary \ref{construct Koszul}, the Koszul relation between $g_2$ and $g_3$ is a minimal first syzygy in $I_1$, which we denote by $s\in K$. Since $f_2 = ug_2 + h$ and $f_3=u'g_3+h'$ for some $u,u' \in R\bs \mf{m}$ and $h,h'\in \mf{m}I_1$, the Koszul relation between $f_2$ and $f_3$ is of the form $vs+t$ where $v\in R\bs \mf{m}$ and $t\in \mf{m}K$. In particular, $f_2$ and $f_3$ yield a minimal Koszul relation in $I_1$, which by Proposition \ref{min Koszul relation}.ii. concludes the proof. 
\end{proof}

We close this section by formulating to what extent we can expect the appearance of minimal Koszul relations in grade 3 licci ideals. There are examples of grade $3$ licci ideals with $\lambda=0$ (see \cite{B87}, Example 4.6). However, there are no known examples of grade $3$ licci ideals with $\lambda=0$ that cannot be minimally linked to an ideal with $\lambda>0$. Also, the only way to link a grade $3$ perfect ideal to a complete intersection is to use a sequence of links in which there are sufficiently many regular sequences with pairs of generators yielding minimal Koszul relations. Indeed, such regular sequences are the only guarantee that the total Betti numbers will decrease in double links. 

\begin{prop} \label{non licci}
Assume $R$ is a local Gorenstein ring, and let $I$ be a grade $3$ perfect ideal with $r(R/I)=r$. If $I$ is not licci, then there is a sequence of links 
\[ I = I_0 \sim I_1 \sim \dots \sim I_n = J \]
where $\lambda(J)=0$ and $n\leq 2(r-2)$. 
\end{prop}

\begin{proof}
If $\lambda(I)>0$, then there is a complete intersection $C\subseteq I$ such that two minimal generators of $C$ yield a minimal Koszul relation in $I$. If $I_1=C:I$, then $d(I_1) < r(R/I)$ by Proposition \ref{min Koszul relation}.ii. If $C_1$ is a complete intersection generated by part of a minimal generating set of $I_1$ and $I_2=C_1:I_1$, then $r(R/I_2) = d(I_1) < r(R/I)$ by Proposition \ref{min Koszul relation}.i. Continuing in this manner, if $\lambda(I_{2t})>0$, we can repeat the steps above to construct an ideal $I_{2(t+1)}$ linked in two steps to $I_{2t}$ with $r(R/I_{2(t+1)})<r(R/I_{2t})$. All grade $3$ perfect $R$-ideals of CM type $1$ are licci (\cite{W73}), so this process must terminate at or before $t=r-2$. 
\end{proof}

Using deformation theory and Gaeta's result that all grade $2$ perfect ideals are licci, it can be deduced that all grade $3$ perfect ideals containing a linear form are licci. However, the following lemma provides an alternate proof using minimal Koszul relations.

\begin{lem} \label{linear Koszul}
Assume $R$ is a local Gorenstein ring, and let $I$ be a nonzero $R$-ideal minimally generated by $f_1, \dots, f_n$. If $\ell$ is a linear form regular on $R/I$, then the Koszul relation between $\ell$ and $f_i$ is a minimal relation in the ideal $(I,\ell)$ for all $i$.
\end{lem}

\begin{proof}
Consider the short exact sequence
\[ 0 \to R/I:\ell(-1) \xrightarrow{\cdot \ell} R/I \to R/(I,\ell) \to 0.\]
Since $\ell$ is regular on $I$, it follows that $I:\ell=I$. Thus, constructing the resolution of $R/(I,\ell)$ as the mapping cone of $R/I:\ell \to R/I$ from the short exact sequence above, the generators of $I:\ell$ correspond to minimal Koszul relations in $(I,\ell)$.
\end{proof}

\begin{cor} \label{g3 linear form}
If $R$ is a local Gorenstein ring, then any grade $3$ perfect $R$-ideal containing a linear form is licci.
\end{cor}

\begin{proof}
If $I$ is a grade $3$ perfect $R$-ideal containing a linear form, then it is possible to construct a complete intersection $C\subseteq I$ containing the linear form. By Lemma \ref{linear Koszul}, the generators of $C$ yield at least two minimal Koszul relations in $I$. Since $C:I$ contains the same linear form, by Lemma \ref{linear Koszul} we have $\lambda(C:I)>0$. From this we see that as long as the linear form is used to construct links, there is an arbitrarily long sequence of links such that $\lambda > 0$ for any ideal in the sequence. Thus $I$ is licci by Proposition \ref{non licci}.
\end{proof}


\section{Special Classes of Grade 3 Licci Ideals}
In this section we recall how the non-minimally licci ideals of \cite{HMNU07} are constructed, and we show they are SBL. Once fixing certain parameters on these ideals, we construct a direct link with a carefully chosen regular sequence to form the class of our non-SBL ideals. There are two technical results concerning two more classes of ideals, which are realized from evenly many sequentially bounded links from our non-SBL ideals. Both of these latter classes of ideals are SBL and contain a quadric, but we show that the SBL property cannot be attained using the quadric. This result is an important step to proving that our main class of ideals are not SBL. 

\begin{defn} \label{SBL}
Suppose $I$ is a homogeneous perfect $R$-ideal linked to a homogeneous complete intersection $I_s$ by 
\[ I=I_0 \widesim{C_0} I_1 \widesim{} \dots \widesim{} I_{s-1} \widesim{C_{s-1}} I_s,\]
where $C_i$ is homogeneous for all $i=0,\dots, s-1$. 
\begin{enumerate}[i.]
\item If $C_i$ has minimal type in $I_i$ for all $i=0,\dots, m-1$, then $I$ is $\textit{minimally licci}$. 
\item If $\text{type}(C_i) \geq \text{type}(C_{i+1})$ for all $i=0,\dots, m-1$, then $I$ is \textit{sequentially bounded licci}. 
\end{enumerate}
\end{defn}

\begin{rem}
With the set-up of Definition \ref{SBL}, we may only be concerned with the types used in the links and write  
\[ I=I_0 \widesim{\text{type}(C_0)} I_1 \widesim{} \dots \widesim{} I_{s-1} \widesim{\text{type}(C_{s-1})} I_s.\]
\end{rem}

\begin{prop} \label{direct link to ML}
Let $I$ be a homogeneous $R$-ideal. If $I$ is directly linked to a minimally licci ideal by a homogeneous regular sequence, then $I$ is sequentially bounded licci. Further, if $I$ has grade $3$ and contains a linear form, then $I$ is minimally licci.
\end{prop}

\begin{proof}
For the first statement, assume $J=C:I$ where $C$ is a homogeneous complete intersection of maximal grade and $J$ is minimally licci. If $D$ is a complete intersection of minimal type in $J$, then $\text{type}(C)\geq \text{type}(D)$. Thus the sequence of minimal links from $J$ to a complete intersection is a sequence of sequentially bounded links from $I$ to a complete intersection. The second statement is a direct consequence of Corollary \ref{g3 linear form}. 
\end{proof}

Now we recall the construction of the ideals by Huneke, Migliore, Nagel, and Ulrich that are homogeneously licci but not minimally licci. 

\begin{thm}[\cite{HMNU07}, Theorem 3.2] \label{HMNU}
Let $c_1, c_2, c_3, c_4$ be integers satisfying $4\leq c_1+3\leq c_2< c_3< c_4$, $c_1\neq 2$, and $c_2+c_3\leq \min \{2,c_1\} + c_4$. Choose homogeneous polynomials $f_1,f_4\in R$ of degrees $c_1, c_4$, respectively, and a linear form $\ell_1$ such that $f_1, f_4, \ell_1$ is a regular sequence. Define $I_1 = (\ell_1, f_1, f_4)$. Choose homogeneous polynomials $f_2, f_3\in I_1$ of degrees $c_2, c_3$, respectively, such that $f_2, f_3$ and $\ell_1, f_2$ form regular sequences, and let $\ell_2$ be a linear form such that $\ell_2, f_2, f_3$ is a regular sequence. Define $I=\ell_2\cdot I_1+(f_2, f_3)$. Then $I$ is homogeneously licci but not minimally licci. 
\end{thm}

Part of their proof shows that the CM type of these ideals is $2$, the minimal type of a complete intersection in $I$ is $(2, c_2, c_4+1)$, and that there is a sequence of links
\[ I \widesim{(c_2, c_3, c_4+1)} I' \widesim{(c_2, c_3, c_4)} I_1. \]
We remark that the drop from $c_4+1$ to $c_4$ in this double link is possible because $f_2$ and $f_3$ yield a minimal Koszul relation in $I$. 

\begin{prop} \label{t2 l1}
If $I$ is a homogeneous grade $3$ perfect $R$-ideal with $r(R/I)=2$ and $\lambda(I)=1$, then $I$ is SBL.
\end{prop} 

\begin{proof}
Let $C\subseteq I$ be a homogeneous complete intersection such that two of the minimal generators yield a minimal Koszul relation in $I$. By Proposition \ref{min Koszul relation}, $d(C:I)\leq 1$. Such ideals are minimally linked to Gorenstein ideals (as seen by Ferrand's mapping cone, for instance), and Gorenstein ideals are minimally licci (\cite{MN08}). Thus $I$ is SBL by Proposition \ref{direct link to ML}.
\end{proof}

If $I$ is an ideal from Theorem \ref{HMNU} and $L$ is a minimal link of $I$, then one possible sequence of links from $L$ to a complete intersection is
\[ L \widesim{(2, c_2, c_4+1)} I \widesim{(c_2, c_3, c_4+1)} I' \widesim{(c_2, c_3, c_4)} I_1 .\]
Since $c_2>2$, the sequence of links from $L$ to the complete intersection $I_1$ are not sequentially bounded. 
\begin{question} \label{diff links}
Is there a different sequence of links from $L$ to a complete intersection which would satisfy the sequentially bounded property?
\end{question}

This question reveals the complexity of trying to show an ideal is \textit{not} SBL. There are infinitely many complete intersection types in $L$ to choose as a starting point, and even once the type is chosen, there are possibly infinitely many complete intersections with that type. Thus giving any answer to the question must rely on some kind of specific structure of the ideal or its minimal graded free resolution. In this case, there is more structure to utilize.

\begin{prop} \label{Koszul in L}
Let $I$ be any ideal from Theorem \ref{HMNU}. If $L$ is any minimal link of $I$, then $r(R/L)=2$ and $\lambda(L)=1$. In particular, $L$ is SBL.
\end{prop}

\begin{proof}
The ideal $I$ from Theorem \ref{HMNU} fits in to a short exact sequence
\[ 0 \to R/I_1(-1) \xrightarrow{\cdot \ell_2} R/I \to R/(\ell_2, f_2, f_3) \to 0 \]
(this statement is in the proof of Theorem \ref{HMNU}). The Horseshoe Lemma can be used to construct the graded free resolution of $I$, including the entries of the differential maps. Further, it is not hard to compute the connecting morphisms induced by an inclusion $C\subseteq I$, where $C$ is a complete intersection of minimal type. From these one obtains an explicit free resolution for $L$ with the differential maps, and the second differential map has precisely one column which corresponds to a Koszul relation. 
\end{proof}

In the construction of Theorem \ref{HMNU}, let $c_1=1$, choose $c_2=n>3$, and let $c_3=n+3$ and $c_4=2n+3$. As shown in the proof of Theorem \ref{HMNU}, the minimal graded free resolution of $I$ is
\begin{equation} \label{I res}
    0 \to \begin{matrix} R(-(2n+4)) \\ \oplus \\ R(-(2n+6)) \end{matrix} \to \begin{matrix} R(-3) \\ \oplus \\ R(-(n+1)) \\ \oplus \\ R(-(n+4)) \\ \oplus \\ R(-(2n+3)) \\ \oplus \\ R(-(2n+5))^2 \end{matrix} \to \begin{matrix} R(-2)^2 \\ \oplus \\ R(-n) \\ \oplus \\ R(-(n+3)) \\ \oplus \\ R(-(2n+4)) \end{matrix} \to R \to 0. 
\end{equation} 
By Lemma \ref{g1q1}, it follows that $g_2(I)=n$. Further, $g_3(I)=2n+4$ by Lemma \ref{max socle shift}. Hence the minimal type of $I$ is $(2,n,2n+4)$, and a direct link of $I$ using this complete intersection type yields an ideal with a minimal graded free resolution
\begin{equation} \label{L res}
0 \to \begin{matrix} R(-(2n+3)) \\ \oplus \\ R(-(3n+4)) \end{matrix} \to \begin{matrix} R(-(n+1))^2 \\ \oplus \\ R(-(n+3)) \\ \oplus \\ R(-(2n+2)) \\ \oplus \\ R(-(2n+5)) \\ \oplus \\ R(-(3n+3)) \end{matrix} \to \begin{matrix} R(-2) \\ \oplus \\ R(-n)^2 \\ \oplus \\ R(-(n+2)) \\ \oplus \\ R(-(2n+4)) \end{matrix} \to R \to 0. 
\end{equation}

\begin{defn} \label{min link res}
By $(\vardiamond)$ we denote the property of an $R$-ideal satisfying the following two conditions for a fixed $n\geq 4$: the ideal has grade $3$, and the Betti table matches the Betti table of an ideal with resolution (\ref{L res}).
\end{defn}

The proof of Theorem \ref{HMNU} in \cite{HMNU07} shows that any ideal satisfying $(\vardiamond)$ is not minimally licci, and we extend this result in Propositions \ref{minimal link of I} and \ref{minimal link of I trim} by showing that the quadratic form is an obstruction to the SBL property.

\par In the remainder of the proofs, one simple observation happens to be quite useful: 

\begin{defn} \label{Koszul argument}
The \textit{Koszul argument} is the observation that if $I$ is an $R$-ideal and $\{f,g\}$ is part of a minimal generating set of $I$ with $\deg f = a \leq b = \deg g$, then $\text{Syz}_1(R/I)_{\leq b} \neq \text{Syz}_1(R/I)_{\leq a+b}$. Indeed, the Koszul syzygy between $f$ and $g$ has degree $a+b$, and since $g$ is part of a minimal generating set, the Koszul syzygy cannot lie in $\text{Syz}_1(R/I_{< b})$, which contains $\text{Syz}_1(R/I)_{\leq b}$. 
\end{defn}


\begin{prop} \label{minimal link of I}
Fix $n\geq 4$ and let $L$ be an ideal satisfying $(\vardiamond)$. If $D\subseteq L$ is a complete intersection with type $(2,c,d)$, then the minimal type of the link $D:L$ is $(2,c,d)$.
\end{prop}

\begin{proof}
First note that $\lambda(L)\leq 1$, the only possible minimal Koszul relation in $L$ being the Koszul relation between the minimal generators of degree $n$ and $n+2$ (in fact, although the proof is omitted, one can show this must be a minimal Koszul relation). Thus no direct link of $L$ is a complete intersection. Let $D\subseteq L$ be a complete intersection of type $(2,c,d)$. Then $L_1=D:L$ has a free resolution $\mc{P}_{\bigcdot}$ with terms
\begin{align} \notag
P_1 &= R(-2) \oplus R(-c) \oplus R(-d) \oplus R(-(c+d-3n-2)) \oplus R(-(c+d-2n-1)) \\ \notag
P_2 &= R(-(c+d-3n-1)) \oplus R(-(c+d-2n-3)) \oplus R(-(c+d-2n)) \\ \notag 
&\oplus R(-(c+d-n-1)) \oplus R(-(c+d-n+1))^2 \oplus R(-(c+2)) \oplus R(-(d+2)) \\ \notag
P_3 &= R(-(c+d-2n-2)) \oplus R(-(c+d-n)) \oplus R(-(c+d-n+2))^2 
\end{align}
Now $g_2(L)\geq n$ and $g_3(L)=2n+4$, the latter a consequence of Lemma \ref{max socle shift}. Therefore, $c+d\geq 3n+4$, with equality if and only if $c=n$ and $d=3n+4$, in which case $\mc{P}_{\bigcdot}$ has the same terms as resolution (\ref{I res}). In this case, the proof of Theorem \ref{HMNU} in \cite{HMNU07} shows that $D$ has minimal type in $L_1$ and any link of $L_1$ with respect to $\text{type}(D)$ reproduces an ideal with the same free resolution as $L$. Thus we may exclude this case from the rest of the proof. However, in any case, we show that $D$ has minimal type in $L_1$ by proving the following two claims. \\


\noindent \textbf{Claim 1:} $g_2(L_1)=c$. \\

\noindent Since $d>2n+1$, adding $c$ to both sides we obtain $c+d-2n-1>c$. Notice there is a syzygy in $L_1$ of degree $c+d-3n-1$, and $c+d-3n-1 \leq c$ with equality only if $d=3n+1$. In particular, if this syzygy is not minimal, then $R(-(c+d-3n-1))$ in $P_2$ trims with $R(-c)$ in $P_1$. In this case, by Proposition \ref{min Koszul relation}, the quadric and degree $3n+1$ generator of $D$ yield a minimal Koszul relation in $L$. However, this is impossible because $\beta_{1,3n+1}(R/L)=0$. Hence the syzygy of degree $c+d-3n-1$ is minimal, and it can only describe a nontrivial, non-Koszul syzygy between the quadratic generator and degree $c+d-3n-2$ generator of $L_1$. Thus $g_2(L_1) > c+d-3n-1$, which means $g_2(L_1)=c$. \\


\noindent \textbf{Claim 2:} $g_3(L_1)=d$. \\

\noindent If $g_3(L_1)<d$, then either $g_3(L_1)=c+d-3n-2<d$ or $g_3(L_1)=c+d-2n-1 < d$. 
\par Assume $g_3(L_1)=c+d-3n-2<d$. Since $g_2(L_1)=c$, we then have $c\leq c+d-3n-2 \leq d-1$. Also, by Lemma \ref{max socle shift}, since $L_1$ is not a complete intersection, $2+2c+d-3n-2 \geq c+d-n+3$. These inequalities imply $d\geq 3n+2$ and $2n+3\leq c\leq  3n+1$. Since $\beta_{1,j}(R/L)=0$ for $j\geq 3n+2$, the generator of degree $d$ in $D$ is not a minimal generator of $L$. Thus, by Corollary \ref{construct Koszul}, there is a complete intersection in $L_1$ of type $(2,c,c+d-3n-2)$ such that the quadric and degree $c$ generators yield a minimal Koszul relation in $L_1$. Let $L'$ be a link of $L_1$ with respect to such a complete intersection. Using Ferrand's resolution of $L'$ and the fact that $c\leq 3n+1$, it is easily checked that $\beta_{1,j}(R/L')=0$ for $j>c$, and $\{j: j>c \text{ and } \beta_{2,j}(R/L')\neq 0 \}$ has one element. This contradicts Lemma \ref{max degree relations}.
\par Now assume $g_3(L_1)=c+d-2n-1 < d$. Since $g_2(L_1)=c$, Lemma \ref{max socle shift} yields inequalities $c\leq c+d-2n-1 \leq d-1$ and $2+2c+d-2n-1 \geq c+d-n+3$. Simplifying, we have
\begin{equation} \label{case 2}
   n+2\leq c\leq  2n  \hspace{5pt} \text{ and } \hspace{5pt} d\geq 2n+1.
\end{equation}
Now, there is a complete intersection of type $(2,c,c+d-2n-1)$ in $L_1$ such that the quadric and degree $c$ generators are weak associates of the generators of the same degree in $D$. Hence by Lemma \ref{ci shift trims}.ii, the link $L_2$ of $L_1$ with respect to this complete intersection has a free resolution $\mc{P}_{\bigcdot}'$ with terms
\begin{align} \notag
P_1' &= R(-2) \oplus R(-(c-n-1))^2 \oplus R(-(c-n+1)) \oplus R(-c) \oplus R(-(c+3)) \\ \notag
P_2' &= R(-(c-n))^2 \oplus R(-(c-n+2)) \oplus R(-(c+1)) \\ \notag
&\oplus R(-(c+4)) \oplus R(-(c+n+2)) \oplus R(-(2c-2n-1)) \\ \notag 
P_3' &= R(-(2c-2n+1)) \oplus R(-(c+n+3)) 
\end{align}
It is easily checked that the conditions of Lemma \ref{g1q1}.ii are met to deduce $g_2(L_2)=c-n-1$. Therefore, $L_2$ contains a complete intersection of type $(2,c-n-1,c+3)$. Applying Lemma \ref{max socle shift} to this type we have $c\geq 2n$, so by (\ref{case 2}), we have $c=2n$. Let $L_3$ be a link of $L_2$ with respect to $(2,c-n-1,c+3)=(2,n-1,2n+3)$. Then $L_3$ has a free resolution $\mc{P}_{\bigcdot}''$ with terms
\begin{align} \notag
P_1'' &= R(-1) \oplus R(-2) \oplus R(-(n-1)) \oplus R(-(n+3)) \oplus R(-(2n+3)) \\ \notag
P_2'' &= R(-2)  \oplus R(-n) \oplus R(-(n+3)) \oplus R(-(n+5)) \\ \notag
&\oplus R(-(2n+2)) \oplus R(-(2n+4))^2 \\ \notag
P_3'' &= R(-(n+4)) \oplus R(-(2n+3)) \oplus R(-(2n+5)) 
\end{align}
By Lemma \ref{linear Koszul}, the terms $R(-2)$ and $R(-(n+3))$ in steps 1 and 2 trim, implying $L_3$ is a complete intersection of type $(1, n-1, 2n+3)$, which contradicts Lemma \ref{max socle shift} since $n\geq 4$. This contradiction finishes the proof of Claim 2. 
\end{proof}

Using that $(2,c,d)$ is the minimal type in $(2,c,d)$, we show in Proposition $\ref{minimal link of I trim}$ that any sequentially bounded double link of $L$ with respect to $(2,c,d)$ satisfies $(\vardiamond)$. Now we turn to the construction of our main class of ideals. 


\subsection{Construction of non-SBL Ideals} Returning to the ideals of Theorem \ref{HMNU}, instead of constructing a minimal link as above, we consider a direct link constructed by a complete intersection that forces exactly one minimal Koszul relation. 
\par Let $I$ be an ideal whose resolution is given by (\ref{I res}). Recall that $g_2(I)=n$, and let $D\subseteq I$ be a complete intersection with type $(2,n+1,2n+4)$. Notice that the degree $n+1$ generator must lie in $\mf{m}I$, so by Corollary \ref{construct Koszul}, the link $D:I$ contains a minimal Koszul relation between the quadric and degree $2n+4$ generators of $D$. Further, the link $D:I$ has a minimal graded free resolution
\begin{equation} \label{J res}
    0 \to \begin{matrix} R(-(2n+4)) \\ \oplus \\ R(-(2n+7)) \\ \oplus \\ R(-(3n+5)) \end{matrix} \to \begin{matrix} R(-(n+2))^2 \\ \oplus \\ R(-(n+4)) \\ \oplus \\ R(-(2n+3)) \\ \oplus \\ R(-(2n+6))^2 \\ \oplus \\ R(-(3n+4)) \end{matrix} \to \begin{matrix} R(-2) \\ \oplus \\ R(-(n+1))^2 \\ \oplus \\ R(-(n+3)) \\ \oplus \\ R(-(2n+4)) \end{matrix} \to R \to 0.
\end{equation}

\begin{defn} \label{main ideals}
By $(\star)$ we denote the property of an $R$-ideal satisfying the following two conditions for a fixed $n\geq 4$: the ideal has grade $3$, and the Betti table matches the Betti table of an ideal with resolution (\ref{J res}).
\end{defn}

Condition $(\star)$ yields the class of ideals that we claim are homogeneously licci but not SBL. Now we prove they are homogeneously licci, but to show they are not SBL, we need the lemmas and propositions developed in the rest of this section, the next section, and Appendix B. 


\begin{thm}
\label{thm:main}
Fix $n\geq 4$. Then any ideal satisfying $(\star)$ is (i.) homogeneously licci but (ii.) not SBL.
\end{thm}

\begin{proof}[Proof of i.]
Let $J$ be an ideal satisfying $(\star)$. First, $g_2(J)=n+1$ by Lemma \ref{g1q1} and $g_3(J)=2n+4$ by Lemma \ref{max socle shift}. Let $C$ be a complete intersection in $J$ of type $(2,n+1,2n+4)$ and set $I'=C:I$. Then $I'$ has a free resolution with the same terms as (\ref{I res}), except possibly with the extra term $R(-(n+1))$ in steps $1$ and $2$. However, these trim as a result of the Koszul argument applied to the quadric and degree $n+1$ generator of $I'$. 
\par Next, consider any choice of minimal generators of $I'$ with degrees $n$ and $n+3$. We claim these yield a minimal Koszul relation in $I'$. Indeed, if not, then all of the Koszul relations between the degree $n+3$ generator and the minimal generators of $I'$ with lower degree lie in $\text{Syz}_1(R/I')_{\leq n+4}$. However, $\beta_{2,n+4}(R/I')=1$, so in this case, by a similar argument to the proof of Lemma \ref{max degree relations}, we have $\grade((I')_{\leq n})=1$. Since $g_2(I')=n$, this proves our claim. Therefore, $\lambda(I')=1$ and $r(R/I')=2$, which implies $I'$, and therefore $J$, is homogeneously licci. 
\end{proof} 

In Section 2, we remarked that any ghost double link of an ideal satisfying $(\star)$ has the same graded resolution, up to certain ghost terms. The following proposition is concerned with a sequence of double links that use the same complete intersection type at each step, but such that we cannot determine if any are ghost double links. The difference is in whether or not generators yielding minimal Koszul relations are being used to construct the links. For ideals satisfying $(\star)$, this situation only occurs if the quadric is used in the first link. The result is our third class of ideals. Although we do not know if there are any ideals in this class that are SBL, we show that the quadratic form is, as in Proposition \ref{minimal link of I}, an obstruction to the SBL property, as long as we restrict the highest generating degree of the first regular sequence. 

\begin{defn} \label{star star}
Fix $n\geq 4$. Suppose $J$ is an ideal satisfying $(\star)$. Then an ideal $H$ satisfies $(\star \star)$ if $\grade(H)=3$ and 
\[ \beta_{i,j}(R/H) = \left \{ \begin{array}{lr} s; & 1\leq i\leq 2, j=2n+2 \\ t; & 1\leq i \leq 2, j=2n+5 \\ \beta_{i,j}(R/J); & \text{else} \end{array} \right. \]
for some nonnegative integers $s$ and $t$.
\end{defn}

\begin{rem} \label{artinian rem}
For the remainder of the paper we assume $N=3$, without the loss of generality. Indeed, if $N>3$, then to obtain the result we could extend to an infinite base field and study the Artinian reduction of the ideals via linear forms. 
\end{rem}


\begin{prop} \label{ghost summands}
Fix $n\geq 4$ and let $H$ be an ideal satisfying $(\star \star)$. With $b\leq 2n+5$, any sequentially bounded double link of $H$ with respect to $(2, a, b)$ satisfies either $(\star \star)$ or $(\vardiamond)$. 
\end{prop} 

\begin{proof}
First, $H$ cannot be a complete intersection since $r(R/H)=3$. Thus, by Lemmas \ref{g1q1} and \ref{max socle shift}, the grade jumps are $g_2(H)=n+1$ and $g_3(H)\geq 2n+4$. In particular, $a\geq n+1$ and $2n+4\leq b\leq 2n+5$. Let $C\subseteq H$ be a complete intersection of type $(2,a,b)$ and let $H_1=C:H$. Then $H_1$ has a free resolution $\mc{G}_{\bigcdot}$ with terms
\begin{align} \notag
G_1 &=  R(-2) \oplus R(-a) \oplus R(-b) \oplus R(-(a+b-3n-3)) \\ \notag
&\oplus R(-(a+b-2n-5)) \oplus R(-(a+b-2n-2)) \\ \notag 
G_2 &= R(-(a+b-3n-2)) \oplus R(-(a+b-2n-4))^2 \oplus R(-(a+b-2n-1)) \\ \label{H1 res}
&\oplus R(-(a+b-n-2))\oplus R(-(a+b-n))^2 \oplus R(-(a+2)) \\ \notag
&\oplus R(-(b+2)) \oplus R(-(a+b-2n-3))^t \oplus R(-(a+b-2n))^s \\ \notag  
G_3 &= R(-(a+b-2n-2)) \oplus R(-(a+b-n-1)) \oplus R(-(a+b-n+1))^2 \\ \notag 
&\oplus R(-(a+b-2n-3))^t \oplus R(-(a+b-2n))^s. 
\end{align}

Since $r(R/H)=3$ and $\lambda(H)\leq 1$, this link cannot be a complete intersection. Now the majority of the proof depends on the next two lemmas and Propositions \ref{H a drop} and \ref{H b drop}. The lemmas show that there are only three complete intersection types in $H_1$ that do not exceed $(2,a,b)$. The first is $(2,a-1,b)$, in which case $b=2n+4$ and this is the minimal type in $H_1$. The second is $(2,a,b-1)$, in which case $a=2n+4$, $b=2n+5$, and this is the minimal type in $H_1$. The third is the repeated type, $(2,a,b)$. Propositions \ref{H a drop} and \ref{H b drop} show that any link of $H_1$ with respect to either of the first two types yields an ideal satisfying $(\vardiamond)$. 
\par Assuming this information, to finish the proof we need to show that if $H_2$ is a link of $H_1$ with respect to $(2, a, b)$, then $H_2$ satisfies $(\star \star)$. The next two lemmas prove this when the minimal type is not $(2,a,b)$. If the minimal type of $H_1$ is $(2,a,b)$, then $H_2$ is a ghost double link of $H$. Let $H_2=C_1:H_1$ where $\text{type}(C_1)=(2,a,b)$. Since the quadric is minimal in $H$, to show $H_2$ satisfies $(\star \star)$, we need to show that for both $j=a$ and $j=b$, either the ghost term $R(-j)$ trims, or $j\in \{2n+2,2n+5\}$. Assume $R(-a)$ does not trim in the induced mapping cone resolution of $H_2$. Then by Lemma \ref{ci shift trims}, the degree $a$ generator of $C$ is not minimal in $H$, and the degree $b$ minimal generator of $C_1$ is not a weak associate of the degree $b$ minimal generator of $C$. However, since $g_3(H_1)=b$, the degree $b$ generator of $C_1$ must be a minimal generator of $H_1$, so $b\in \{a+b-3n-3, a+b-2n-5, a+b-2n-2\}$. Therefore, $a\in \{3n+3, 2n+5, 2n+2\}$, but $a\leq b<3n+3$, so $a\in \{2n+5, 2n+2\}$. The same argument applies with the roles of $a$ and $b$ switched. Hence $H_2$ satisfies $(\star \star)$. 
\end{proof}


\begin{lem} \label{ghost g2}
With the same set-up as Proposition \ref{ghost summands}, let $H_1=C:H$ where $C\subseteq H$ has type $(2,a,b)$. If $g_2(H_1) < a$, then $g_2(H_1) = a-1$ and $g_3(H_1) = b=2n+4$. Further, in this case, any link of $H_1$ with respect to $(2,a,b)$ satisfies $(\star \star)$.
\end{lem}

\begin{proof}
Recall that the terms of a graded free resolution of $R/H_1$ are given by (\ref{H1 res}) above, and that $2n+4\leq b \leq 2n+5$. Then $$a+b-2n-2 \geq a+2 > a \geq a+b-2n-5 > a+b-3n-2,$$ so the minimal syzygy of degree $a+b-3n-2$ is from a linear relation between the quadric and degree $a+b-3n-3$ generator of $H_1$. Since $a+b-3n-3\geq 2$, if $g_2(H_1)<a$, then $g_2(H_1) = a+b-2n-5$. In this case, we have $b<2n+5$, which implies $b=2n+4$ and $g_2(H_1)=a-1$. 
\par Since $b=2n+4$, the minimal generator of $C$ of degree $b$ must be a minimal generator of $H$. Hence, by Proposition \ref{min Koszul relation}.i, the term $R(-(a+2))$ in step 2 trims with $R(-(a+b-2n-2))$ in step 3 in resolution (\ref{H1 res}). Substituting $b=2n+4$ into the terms of resolution (\ref{H1 res}), one sees that either $\text{Syz}(R/H_1)_{\leq a+2} = \text{Syz}(R/H_1)_{\leq a}$ or $a=2n+4$. In the former case, the Koszul argument implies $R(-a)$ in step 1 trims with $R(-(a+b-2n-4))$ in step 2. In the case that $a=2n+4$, since $g_2(H_1)=2n+3$ and $\beta_{1,2n+3}(R/H_1)=1$, we have $\grade((H_1)_{\leq b-1})=2$, so $g_3(H_1)=b$.
\par Now assume $a<b=2n+4$ and $g_3(H_1)<b$. The only possibility is that $g_3(H_1) = a+2$. The link $H'$ of $H_1$ with respect to $(2,a-1,a+2)$ has a resolution $\mc{G}_{\bigcdot}'$ with terms
\begin{align} \notag
G_1' &= R(-2) \oplus R(-(a-n-2))^2 \oplus R(-(a-n))\oplus R(-(a-1))^{s+1} \oplus R(-(a+2))^{t+1} \\ \notag
G_2' &= R(-(a-n-1))^2 \oplus R(-(a-n+1)) \oplus R(-a) \oplus R(-(a+3)) \\ \notag
&\oplus R(-(a+n+1)) \oplus R(-(2a-2n-3)) \oplus R(-(a+2))^t \oplus R(-(a-1))^s \\ \notag
G_3' &= R(-(2a-2n-1)) \oplus R(-(a+n+2)).
\end{align} 
It is impossible to trim $\mc{G}_{\bigcdot}'$ down to a resolution which could describe a complete intersection (there are not enough ghost terms), so $H'$ cannot be a complete intersection. Further, since $g_3(H_1)<b$ and $g_3(H_1)=a+2$, we have $a\leq 2n+1$. Therefore, 
\[ 2(a-n-2) \leq a-1 \text{   and   } 2(a-n-2)\leq 2a-2n-3.\] 
Hence Lemma \ref{g1q1} implies $H'$ contains a complete intersection of type $(2,a-n-2,a+2)$. But since $H'$ is not a complete intersection and $a+n+2\geq 2a-2n-1$, Lemma \ref{max socle shift} implies $$2a-n+2\geq a+n+2.$$ Then $2n+1\leq a\leq 2n+1$, so $a=2n+1$. The link $H''$ of $H'$ with respect to $(2,a-n-2,a+2) = (2,n-1,2n+3)$ has a resolution $\mc{G}_{\bigcdot}''$ with terms
\begin{align} \notag
G_1'' &= R(-1) \oplus R(-2) \oplus R(-(n-1))\oplus R(-(n+3)) \oplus R(-(2n+3)) \\ \notag
G_2'' &= R(-2) \oplus R(-n) \oplus R(-(n+1))^t \oplus R(-(n+3)) \\ \notag
&\oplus R(-(n+4))^s \oplus R(-(n+5)) \oplus R(-(2n+2)) \oplus R(-(2n+4))^2 \\ \notag
G_3'' &= R(-(n+1))^t \oplus R(-(n+4))^{s+1} \oplus R(-(2n+3)) \oplus R(-(2n+5)).
\end{align} 
However, there are no ideals with such a graded free resolution. Indeed, since $n\geq 3$, the Koszul argument implies $R(-2)$ trims, and since $(H'')_{<n+3}$ is generated by a regular sequence of length two, the generators contribute exactly one minimal syzygy to $R/H''$. Thus $R(-(n+3))$ trims. But then $d(H'')=0$, so $H''$ is a complete intersection of type $(1,n-1,2n+3)$, which contradicts, among other properties, Lemma \ref{max socle shift}. This contradiction started with the assumption that $g_3(H_1)<b$, so $g_3(H_1)=b$.  
\par Therefore, the minimal type of $H_1$ is $(2,a-1,b)=(2,a-1,2n+4)$. To conclude the proof, suppose $H_2=C_1:H_1$ where $\text{type}(C_1)=(2,a,2n+4)$. Then the induced mapping cone resolution of $H_2$ has the same terms as $H$, plus $R(-a)\oplus R(-(2n+4))$ in steps 1 and 2. Following the proof of Lemma \ref{ci shift trims}.i, since $g_3(H_1)=2n+4$, we get that $R(-(2n+4))$ trims, and if $R(-a)$ does not trim, then the degree $a$ generator of $C$ is not minimal in $H$ and the degree $b$ generator of $C_1$ is not a weak associate of the degree $b$ generator of $C$. But $g_3(H_1)=b$, so $\beta_{1,b}(R/H_1)\neq 0$, and the only possibility is that $b=a+2$. In this case, $a=2n+2$, so $H_2$ satisfies $(\star \star)$.
\end{proof}



\begin{lem} \label{ghost g3}
With the same set-up as Proposition \ref{ghost summands}, let $H_1=C:H$ where $C\subseteq H$ has type $(2,a,b)$. If $g_3(H_1)<b$, then $a=2n+4$, $b=2n+5$, and $g_3(H_1)=2n+4$. Further, in this case, any link of $H_1$ with respect to $(2,a,b)$ satisfies $(\star \star)$.
\end{lem}

\begin{proof} 
Recall that the terms of a graded free resolution of $R/H_1$ are given by (\ref{H1 res}) above. By the previous lemma we may assume $g_2(H_1)=a$.  
\par Suppose $g_3(H_1)<b$. Then since $b \leq 2n+5 < 3n+3$, we know $a+b-3n-3<a$. Thus $g_3(H_1)\neq a+b-3n-3$, which leaves only two cases: either $g_3(H_1) = a+b-2n-2$ or $g_3(H_1)=a+b-2n-5$. We show the first case is impossible, and that the second case only happens under the conditions in the statement of the proposition. \\


\noindent \textbf{Case 1:} $g_3(H_1) = a+b-2n-2<b$. In this case, $a\leq 2n+1$. Let $C_1\subseteq H_1$ be a complete intersection of type $(2,a,a+b-2n-2)$, and let $H_2=C_1:H_1$. Then $H_2$ has a free resolution $\mc{G}_{\bigcdot}'$ with terms
\begin{align} \notag
G_1' &= R(-2) \oplus R(-(a-n-1))^2 \oplus R(-(a-n+1)) \oplus R(-a)^{s+1}  \\ \notag
&\oplus R(-(a+2)) \oplus R(-(a+3))^t  \\ \label{H2 2a-2n trim} 
G_2' &= R(-(a-n))^2 \oplus R(-(a-n+2)) \oplus R(-(a+1)) \oplus R(-(a+4))^2 \\ \notag
&\oplus R(-(a+n+2)) \oplus R(-(2a-2n-2)) \oplus R(-a)^s \oplus R(-(a+3))^t  \\ \notag
G_3' &= R(-(2a-2n)) \oplus R(-(a+5)) \oplus R(-(a+n+3)).
\end{align}
Notice the absence of $R(-(a+b-2n-2))$ in steps 1 and 2, since they trim by Lemma \ref{ci shift trims}.ii. 
\par Now $a\leq 2n+1$, so $2(a-n-1)< a$, which by Lemma \ref{g1q1} means $g_2(H_2)=a-n-1$. Then there is a complete intersection $C_2\subseteq H_2$ with type $(2,a-n-1,a+3)$, and applying Lemma \ref{max socle shift} to this type, we have $a\geq 2n$. Thus either $a=2n$ or $a=2n+1$. 
\par If $a=2n$, then $H_3=C_2:H_2$ contains a linear form. Lemma \ref{linear Koszul} implies $R(-(n+4))=R(-((2a+b-2n)-(2a-2n))$ in steps 1 and 2 trim in the Ferrand's resolution of $H_3$. We trace this unit entry back to the resolution of $H_1$. Indeed, By Proposition \ref{min Koszul relation}.ii, the terms $R(-a)$ and $R(-(2a-2n))$ in steps 3 and 4, respectively, of (\ref{H2 2a-2n trim}) could trim, and by Proposition \ref{min Koszul relation}.i, the terms $R(-b)$ and $R(-(a+b-2n))^s$ in steps 1 and 2, respectively, of (\ref{H1 res}) could trim. Again by Proposition \ref{min Koszul relation}.ii, the quadric and degree $a$ generator of $C$ yield a minimal Koszul relation in $H$, which in particular implies the degree $a$ generator of $C$ is a minimal generator of $H$. However, $a=2n$ and $\beta_{1,2n}(R/H)=0$, so we have a contradiction.
\par If $a=2n+1$, the link $H_3=C_2:H_2$ has a free resolution $\mc{G}_{\bigcdot}''$ with graded terms
\begin{align} \notag
G_1'' &= R(-2)^2 \oplus R(-n)^2 \oplus R(-(n+4)) \oplus R(-(2n+4)) \\ \label{H3 res imp}
G_2'' &= R(-3) \oplus R(-(n+1))^2 \oplus R(-(n+4)) \oplus R(-(2n+3)) \\ \notag 
&\oplus R(-(2n+5))^2 \oplus R(-(n+6)) \oplus R(-(n+2))^t \oplus R(-(n+5))^s \\ \notag
G_3'' &= R(-(n+3)) \oplus R(-(2n+4))\oplus R(-(2n+6)) \\ \notag
&\oplus R(-(n+2))^{t-1} \oplus R(-(n+5))^{s+1}.
\end{align}
First notice that if $R(-(n+4))$ trims, then by the same argument above, we can trace back the unit entry to (\ref{H1 res}) and conclude that the degree $a=2n+1$ generator of $C$ is a minimal generator of $H$, which is impossible since $\beta_{1,2n+1}(R/H)=0$. This impossibility proves that if $g_3(H_1)<b$, then $g_3(H_1)\neq a+b-2n-2$, by the following argument. 
\par Let $$A=(H_3)_{\leq n}=(q_1, q_2, f_1, f_2),$$ where $\deg q_i=2$ and $\deg f_i=n$ for $i=1,2$, and set $$B=(q_1, q_2, f_1)$$ where $f_1$ is chosen so that $\grade(B)=2$. The degree $3$ syzygy in $H_3$ implies there are three linear forms $\ell,\ell_1$, and $\ell_2$ such that $\ell_1$ and $\ell_2$ are linearly independent and $q_i = \ell \ell_i$ for $i=1,2$. Notice that $(B,\ell)=(\ell,f_1)$ is a grade $2$ complete intersection. Then using the Horseshoe lemma on the short exact sequence
\begin{equation} \label{SES ha}
0 \to R/(B:\ell) (-1) \xrightarrow{\cdot \ell} R/B \to R/(B,\ell) \to 0,
\end{equation}
and (\ref{H3 res imp}) to construct the minimal graded resolution of $R/B$, we have $\beta_{1,n}(R/(B:\ell))=1$. However, $B:\ell\supseteq (\ell_1,\ell_2,f_1)$, so $B:\ell=(\ell_1,\ell_2,f_1)$ is a grade $3$ complete intersection.
\par Now the mapping cone induced by the first map in the short exact sequence
\[ 0 \to R/(B:f_2)(-n) \xrightarrow{\cdot f_2} R/B \to R/A \to 0 \]
yields a graded free resolution of $R/A$. Using (\ref{H3 res imp}) above,  $\beta_{2,n+4}(R/A)>0$, and we also know $\beta_{2,n+4}(R/B)=0$ by (\ref{SES ha}), so $\beta_{1,4}(R/(B:f_2))>0$. Since $B:f_2$ also contains two linearly independent linear forms, it follows that $B:f_2$ is a grade $3$ complete intersection with a $\beta_{3,6}(R/B:f_2)=1$. But then $\beta_{4,n+6}(R/A)>0$, which implies pd$(R/A)=4$. This is a contradiction (recall Remark \ref{artinian rem}).   \\


\noindent \textbf{Case 2:} $g_3(H_1) = a+b-2n-5<b$. Then $a\leq a+b-2n-5<b$ implies $b\geq 2n+5$, so $b=2n+5$. In this case $a+b-2n-5=a$, so the minimal type in $H_1$ is $(2,a,a)$. If $H_2$ is a minimal link of $H_1$, then $H_2$ has a free resolution $\mc{G}_{\bigcdot}''$ with terms
\begin{align} \notag 
G_1'' &= R(-2) \oplus R(-(a-n-4))^2 \oplus R(-(a-n-2)) \oplus R(-(a-3))^s \\ \notag
&\oplus R(-(a-1)) \oplus R(-a)^{t+2} \\ \label{H2 res}
G_2'' &= R(-(a-n-3))^2 \oplus R(-(a-n-1)) \oplus R(-(a-2)) \oplus R(-(a+1))^2 \\ \notag
&\oplus R(-(a+n-1)) \oplus R(-(2a-2n-5)) \oplus R(-(a-3))^{s} \oplus R(-a)^{t+1} \\ \notag
G_3'' &= R(-(2a-2n-3)) \oplus R(-(a-1)) \oplus R(-(a+n)).
\end{align}
Since $a<b=2n+5$, we have $2(a-n-4) \leq a-3$. Then Lemma \ref{g1q1} guarantees there is a complete intersection in $H_2$ of type $(2, a-n-4, a)$. Since $\beta_{2,a-n-3}(R/H_2)=2$, there are two linear relations among the generators of $H_2$, and $H_2$ does not contain any linear forms. Hence $H_2$ is not a complete intersection. Therefore, applying Lemma \ref{max socle shift} to the type $(2,a-n-4,a)$ we have $a\geq 2n+3$, so either $a=2n+3$ or $a=2n+4$. In the former case, the link of $H_2$ with respect to $(2, a-n-4, a)$ has a linear form. Using Lemma \ref{linear Koszul} and Proposition \ref{min Koszul relation} to trace back unit entries in the maps of the resolutions above, we get that the degree $a=2n+3$ generator of $C$ is minimal in $H$, which is impossible since $\beta_{1,2n+3}(R/H)=0$. Thus $a=2n+4$. \\

To summarize, Case 1 is impossible, and Case 2 implies $a=2n+4$ and $b=2n+5$. The proof of Proposition \ref{H b drop} shows that $s=t=0$ and $R(-b)$ trims in steps 1 and 2 of (\ref{H1 res}). Then Proposition \ref{min Koszul relation}.ii implies the quadric and degree $a$ generator of $C$ yield a minimal Koszul relation in $H$, so the degree $a$ generator of $C$ is a minimal generator of $H$. Thus $R(-(a+b-2n-2))$ in step 3 trims with $R(-(b+2))$ in step 2 of (\ref{H1 res}). With this information, if $H_2$ is a link of $H_1$ with respect to $(2,a,b)=(2,2n+4,2n+5)$, then Ferrand's induced resolution of $H_2$ shows that $H_2$ satisfies $(\star \star)$, with $\beta_{1,2n+2}(R/H_2)=\beta_{2,2n+2}(R/H_2)=0$ and $\beta_{1,2n+5}(R/H_2)=\beta_{2,2n+5}(R/H_2)\leq 1$. 
\end{proof}




\section{Proof of the Main Theorem}
In this section we finish the proof of our main result, Theorem \ref{thm:main}. The proof of the theorem depends on how the minimal type of a link of an ideal satisfying $(\star)$ compares to the type of the complete intersection used to construct the link. We see that they are the same if the quadric is avoided, and that it could only slightly change when the quadric is used. In any case, sequentially bounded double links of ideals satisfying $(\star)$ will lie in one of the three classes defined in the previous section: Recall Definitions \ref{min link res}, \ref{main ideals}, and \ref{star star} for  $(\vardiamond)$, $(\star)$, and $(\star \star)$, respectively. 

\begin{proof}[Proof of Theorem \ref{thm:main}.ii]
Fix $n\geq 4$ and let $J$ be an ideal satisfying $(\star)$. In Theorem \ref{thm:main}.i, we showed $J$ is homogeneously licci, and now we show $J$ is not SBL. Let $C\subseteq J$ be a complete intersection with type $(a_1, a_2, a_3)$.
\begin{enumerate}[(1.)]
\item If $a_1>2$, then any sequentially bounded double link of $J$ with respect to $(a_1,a_2,a_3)$ satisfies $(\star)$. This follows from Propositions \ref{without quadratic} and \ref{big a1 trim}.
\item If $a_1=2$, then any sequentially bounded double link of $J$ with respect to $(a_1,a_2,a_3)$ satisfies either $(\vardiamond)$ or $(\star \star)$. When $a_3\leq 2n+5$, this is proved in Proposition \ref{ghost summands} (since the conditions of $(\star)$ meet the conditions of $(\star \star)$). When $a_3\geq 2n+6$, this is proved in Proposition \ref{with quadratic}.
\end{enumerate}
Therefore, if there is a sequence of sequentially bounded links
\[ J = J_0 \sim \dots \sim J_{2s} = J',\]
then with (1.) and (2.) above and Proposition \ref{minimal link of I}, it follows that $J'$ satisfies either $(\star)$, $(\star \star)$, or $(\vardiamond)$. Therefore, $r(R/J')=3$ and $\lambda(J')\leq 1$. Moreover, $J'$ is not a complete intersection, and cannot be directly linked to a complete intersection (otherwise, $\lambda(J')\geq 3$). Therefore, $J$ is not SBL. 
\end{proof}


\begin{prop} \label{without quadratic}
Fix $n\geq 4$ and let $J$ be an ideal satisfying $(\star)$. If $C\subseteq J$ is a complete intersection of type $(a_1,a_2,a_3)$ with $a_1>2$, then the minimal type of the link $C:J$ is $(a_1,a_2,a_3)$. 
\end{prop}

\begin{proof}
Let $a=a_1+a_2+a_3$. The link $J_1 = C:J$ has a graded free resolution $\mc{F}_{\bigcdot}'$ with graded free summands
\begin{align} \notag
F_1' &= R(-(a-3n-5)) \oplus R(-(a-2n-7)) \oplus R(-(a-2n-4))  \\ \notag
&\oplus R(-a_1) \oplus R(-a_2) \oplus R(-a_3) \\ \label{J1 res no quad}
F_2' &= R(-(a-3n-4)) \oplus R(-(a-2n-6))^2 \oplus R(-(a-2n-3)) \\ \notag
&\oplus R(-(a-n-4)) \oplus R(-(a-n-2))^2 \oplus R(-(a_1+a_2)) \\ \notag
& \oplus R(-(a_1+a_3)) \oplus R(-(a_2+a_3)) \\ \notag
F_3' &= R(-(a-2n-4) \oplus R(-(a-n-3)) \\ \notag
&\oplus R(-(a-n-1))^2 \oplus R(-(a-2)).
\end{align} 
Since $a_1>2$, the summand $R(-(a-2))$ in step 3 cannot trim and is the maximal shift in the resolution. Assume $C_1\subseteq J_1$ is a complete intersection with $\text{type}(C_1)< \text{type}(C)$. Then by Lemma \ref{max socle shift}, since $a-3n-5\geq a_1$, we have $\text{type}(C_1)=(a_1, a_2', a_3')$ where $a_2+a_3=a_2'+a_3'-1$ and $a_i'=a_i-1$ for either $i=2$ or $i=3$. The link $J_2=C_1:J_1$ has a presentation $F_2\to F_1\to J_2$, where
\begin{align} \notag
F_1 &= R(-1)\oplus R(-n)^2 \oplus R(-(n+2)) \oplus R(-(2n+3)) \oplus R(-a_1) \oplus R(-a_2')\oplus R(-a_3') \\ \notag
F_2 &= \oplus R(-(n+1))^2\oplus R(-(n+3))\oplus R(-(2n+2))\oplus R(-(2n+5))^2 \oplus R(-(3n+3)) \\ \notag 
&\oplus R(-(a_1-1))\oplus R(-(a_2-1))\oplus R(-(a_3-1)) \notag
\end{align}
Assume $a_2'=a_2-1$. Then by Lemma \ref{ci shift trims}.i., to reach a contradiction, we may assume that $R(-(a_2-1))$ trims. Since $C_1$ is the minimal type of $J_1$, we have $\beta_{1,a_2-1}(R/J_1)>0$, so $a_2-1\in \{a-3n-5, a-2n-7, a-2n-4\}$. Therefore, 
\begin{equation} \label{step 1 contr} 
a_1+a_3\in \{ 3n+4, 2n+6, 2n+3\} 
\end{equation}
However, the Koszul argument applied to the linear form and degree $2n+3$ generator of $J_2$ yields the inequality 
\[  
2n+3 \leq a_i-1\leq 2n+4
\]
for either $i=1$ or $i=3$. By (\ref{step 1 contr}), since $a_1\leq a_3$, we must have $i=3$. Then $2n+4\leq a_3\leq 2n+5$, and by (\ref{step 1 contr}) with the fact that $a_1>2$, we get $n-1\leq a_1\leq n$. Then $\beta_{2,a_1-1}(R/J_2)=1$, which is impossible since $\sum_{j\leq n-1} \beta_{1,j}(R/J_2)=1$. The same argument applies if $a_3'=a_3-1$. Hence $\text{type}(C)=(a_1,a_2,a_3)$ is the minimal type of $J_1$.
\end{proof}

Using Proposition \ref{without quadratic}, when $a_1>2$, we see that any sequentially bounded double link of $J$ with respect to $(a_1,a_2,a_3)$ is actually a ghost double link of $J$. Proposition \ref{big a1 trim} shows that the ghost terms $R(-a_i)$ all trim, which proves step (1.) above. To prove step (2.), we first need two lemmas that detect and describe using the minimal Koszul relation in $J$ to construct a link. 

\begin{lem} \label{deduce trim}
Fix $n\geq 4$ and let $J$ be an ideal satisfying $(\star)$. Let $C\subseteq J$ be a complete intersection with type $(2,a,b)$ and $b\geq 2n+6$, and set $J_1=C:J$. If $g_3(J_1)<b$, then the quadric and degree $a$ generators of $C$ yield a minimal Koszul relation in $J$.  
\end{lem}

\begin{proof}
The link $J_1=C:J$ has a free resolution $\mc{F}_{\bigcdot}'$ with graded free modules
\begin{align} \notag
F_1' &= R(-2) \oplus R(-a) \oplus R(-b) \\ \notag
&\oplus R(-(a+b-3n-3)) \oplus R(-(a+b-2n-5)) \oplus R(-(a+b-2n-2)) \\ \label{J1 res res}
F_2' &= R(-(a+b-3n-2)) \oplus R(-(a+b-2n-4))^2 \oplus R(-(a+b-2n-1)) \\ \notag
&\oplus R(-(a+b-n-2)) \oplus R(-(a+b-n))^2 \oplus R(-(a+2)) \oplus R(-(b+2)) \\ \notag
F_3' &= R(-(a+b-2n-2)) \oplus R(-(a+b-n-1)) \oplus R(-(a+b-n+1))^2.  
\end{align}
First notice that by Proposition \ref{min Koszul relation}.ii, $\beta_{1,a}(R/J_1)\geq 1$. Indeed, the quadric and degree $b$ generator of $C$ do not yield a minimal Koszul relation in $J$. 
\par If $g_2(J_1)<a$, then since $b\geq 2n+6$, we must have $g_2(J_1) = a+b-3n-3$. In this case, $2n+6\leq b\leq 3n+2$. Then $a+b-3n-2 \leq a$, so there is a minimal, non-Koszul syzygy between the quadric and degree $a+b-3n-3$ generators, which is a contradiction. Therefore, $g_2(J_1) = a$. 
\par Assume $g_3(J_1)<b$. Then $g_3(J_1)\in \{a+b-3n-3,a+b-2n-5,a+b-2n-2\}$, and we proceed to show $g_3(J_1)=a+b-2n-5$ is the only option. Let $C_1\subseteq J_1$ have type $(2,a,g_3(J_1))$ and set $J_2=C_1:J_1$. \\

\noindent \textbf{Case 1:} $b>g_3(J_1)=a+b-3n-3\geq a$. Then $a\leq 3n+2$, so $2a-3n-3 \leq a$. Therefore, using Ferrand's induced resolution of $J_2$, we see that if $j>a$ and $\beta_{2,j}(R/J_2)>0$, then $j=a+1$ and $\beta_{2,a+1}(R/J_2)=1$. However, $\beta_{1,a}(R/J_2)=1$ and $\beta_{1,j}(R/J_2)=0$ for $j>a$, so $\grade(J_2)=2$ by Lemma \ref{max degree relations}. This is a contradiction. \\
    
\noindent \textbf{Case 2:} $b>g_3(J_1)=a+b-2n-2\geq a$. Then 
\begin{equation} \label{a > 2n+1}
    a\leq 2n+1,
\end{equation} 
and $J_2$ has a free resolution $\mc{F}_{\bigcdot}''$ with terms
\begin{align} \notag
F_1'' &= R(-2) \oplus R(-(a-n-1))^2 \oplus R(-(a-n+1)) \oplus R(-a) \oplus R(-(a+2)) \\ \label{case 2 a1 2}
F_2'' &= R(-(a-n))^2 \oplus R(-(a-n+2)) \oplus R(-(a+1)) \\ \notag
&\oplus R(-(a+4))^2 \oplus R(-(a+n+2)) \oplus R(-(2(a-n-1))) \\ \notag 
F_3'' &= R(-(a+5)) \oplus R(-(a+n+3)) \oplus R(-(2(a-n))).
\end{align} 
Notice that $R(-(a+b-2n-2))$ does not appear in the resolution, since to get a contradiction we may apply Lemma \ref{ci shift trims}. 
\par We now claim $g_2(J_2)=a-n-1$. As long as $a\neq n+5$, this is true by Lemma \ref{g1q1}. If $a=n+5$ and $g_2(J_2)> a-n-1=4$, then after substituting $a=n+5$ into (\ref{case 2 a1 2}) we see that there is a complete intersection $C_2\subseteq J_2$ with type $(2,6,n+7)$ such that the quadric yields a minimal Koszul relation with the other two generators. Then by Proposition \ref{min Koszul relation}.ii, the link $J_3=C_2:J_2$ has a resolution 
\[ 0\to \begin{matrix} R(-10) \\ \oplus \\ R(-(n+11))^2 \end{matrix} \to \begin{matrix} R(-(8-n)) \\ \oplus \\ R(-6) \\ \oplus \\ R(-9) \\ \oplus \\ R(-(n+8)) \\ \oplus \\ R(-(n+10))^2 \end{matrix} \to \begin{matrix} R(-2) \\ \oplus \\ R(-(7-n)) \\ \oplus \\ R(-5) \\ \oplus \\ R(-(n+5)) \end{matrix} \to J_3 \to 0. \]
Applying the Koszul argument to the quadric and degree $n+5$ generator of $J_3$, we get $n=4$ and the term $R(-(n+5))$ in step 1 trims with $R(-9)$ in step 2. Then $J_3$ is a complete intersection of type $(2,7-n=3,5)$, which is impossible by Lemma \ref{max socle shift} since $\beta_{2,15}(R/J_3)=0$. Therefore, $g_2(J_2) = a-n-1$. 
\par Now let $C_2\subseteq J_2$ be a complete intersection with type $(2,a-n-1,a+2)$. By Lemma \ref{max socle shift}, since $J_2$ cannot be a complete intersection, $2a-n+3 \geq a+n+4$, so $a\geq 2n+1$. By (\ref{a > 2n+1}), we get $a=2n+1$. Thus the link $J_3=C_2:J_2$ has a presentation $F_2''' \to F_1'''\to J_3 $ where
\begin{align} \notag
F_1''' &= R(-1) \oplus R(-2) \oplus R(-(n-1))\oplus R(-n) \oplus R(-(n+3)) \oplus R(-(2n+3)), \\ \notag
F_2''' &= R(-2) \oplus R(-n)^2 \oplus R(-(n+3)) \oplus R(-(n+5)) \oplus R(-(2n+2)) \oplus R(-(2n+4))^2. \notag
\end{align}
Since $\sum_{j<2} \beta_{1,j}(R/J_3)=1$, the term $R(-2)$ in steps 1 and 2 trims. Since the linear form and degree $n-1$ generator contribute only one minimal syzygy to $J_3$, the terms $R(-n)$ trim. Moreover, the Koszul argument shows that $R(-(n+3))$ trims. Therefore, $J_3$ is a complete intersection of type $(1,n-1,2n+3)$ with $\beta_{3,3n+3}(R/J_3)=0$, which is impossible by Lemma \ref{max socle shift}. \\

\noindent \textbf{Case 3:} $b>g_3(J_1)=a+b-2n-5\geq a$. By this inequality and Lemma \ref{max socle shift},
\begin{equation} \label{ a < 2n+4}
n+5\leq a\leq 2n+4, 
\end{equation} 
and $J_2$ has a free resolution $\mc{F}_{\bigcdot}''$ with terms
\begin{align} \notag
F_1'' &= R(-2) \oplus R(-(a-n-4))^2 \oplus R(-(a-n-2)) \oplus R(-(a-1)) \\ \notag 
&\oplus R(-a) \oplus R(-(a+b-2n-5)) \\ \label{F2 element}
F_2'' &= R(-(a-n-3))^2 \oplus R(-(a-n-1)) \oplus R(-(a-2)) \oplus R(-(a+1))^2 \\ \notag
&\oplus R(-(a+n-1)) \oplus  R(-(2a-2n-5)) \oplus R(-(a+b-2n-5)) \\ \notag
F_3'' &= R(-(2a-2n-3)) \oplus R(-(a-1)) \oplus R(-(a+n)).
\end{align}  
If $R(-(a+b-2n-5))$ does not trim, then by Lemma \ref{ci shift trims}, the degree $a$ generator of $C_1$ is not a weak associate of the degree $a$ generator of $C$. Since $b\geq 2n+6$ and the degree $a$ generator of $C_1$ must be a minimal generator of $J_1$, this is only possible if $a+b-3n-3=a$, which means $b=3n+3$. Substituting $b=3n+3$ into (\ref{J1 res res}), we can write $(J_1)_{\leq a}=(q,f_1,f_2)$ where $\deg q=2$ and $\deg f_i=a$ for $i=1,2$. Let $\{e_1,e_2,e_3\}$ be part of a basis for $F_1'$ that corresponds to $q,f_1,f_2$, respectively. Since the degree $b$ generator of $C$ cannot be minimal in $J$, we may assume $q,f_1$ yield a minimal Koszul relation in $J_1$. Then $\text{Syz}_1(R/J_1)_{\leq a+2}$ can be generated by two syzygies $s_1$ and $s_2$, where $$s_1=ge_1-\ell_1 e_2-\ell_2 e_3$$ and $$ s_2 = f_1e_1-qe_2.$$ Then
$$ \text{Syz}_1(R/J_1)_{\leq a+2}\ni f_2 e_1 - qe_3 = \ell s_1 + u s_2,$$
where $\ell$ is a linear form and $u$ is a unit in $R$. Then 
$$ q = \ell \ell_2 = -u \ell \ell_1,$$
and $g\ell = f_2 - u^{-1}f_1$. Thus if $f_2'=f_2-u^{-1}f_1$, then
\[ (J_1)_{\leq a}=(q,f_1,f_2'),\]
and $\ell_2 \in q:f_2'$. Thus we conclude that the degree $a$ generator of $C_1$ is a weak associate of the degree $a$ generator of $C$, which implies $R(-(a+b-2n-5))$ trims in (\ref{F2 element}). 
\par Now if $a=n+5$, then $J_2$ contains two linearly independent linear forms, and therefore must be a complete intersection. However, (\ref{F2 element}) cannot trim to a resolution describing a complete intersection. Thus $a\geq n+6$. We can now apply Lemma \ref{g1q1} to get $g_2(J_2)=a-n-4$, so $J_2$ contains a complete intersection $C_2$ with type $(2, a-n-4, a)$. The link $J_3=C_2:J_2$ has a free resolution $\mc{F}_{\bigcdot}'''$ with terms
\begin{align} \notag
F_1''' &= R(-2) \oplus R(-(a-2n-2)) \oplus R(-(a-n-1)) \oplus R(-(n+1)) \\ \notag
&\oplus R(-(a-n-4)) \oplus R(-a) \\ \notag 
F_2''' &= R(-(a-2n-1)) \oplus R(-(a-n-3))^2 \oplus R(-(a-n))\oplus R(-(a-1)) \\ \notag
& \oplus R(-(a+1))^2 \oplus R(-(n+3)) \\ \notag
F_3''' &= R(-(a-n-1)) \oplus R(-a) \oplus R(-(a+2)).
\end{align}
Applying Lemma \ref{max socle shift} to $\text{type}(C_2)$ in $J_2$, we have $a\geq 2n+3$. If $a=2n+3$, the Koszul argument yields a contradiction when considering the generators of degrees $a-2n-2=1$ and $n+1$. Thus by (\ref{ a < 2n+4}), we have $a=2n+4$. Lemma \ref{g12l} implies $R(-(n+1))$ trims in steps 1 and 2 of $\mc{F}_{\bigcdot}'''$, which implies $R(-(2a-2n-3))$ and $R(-(a+1))$ trim in (\ref{F2 element}), which implies the summand $R(-b)$ trims with $R(-(a+b-2n-4))$ in (\ref{J1 res res}). Thus the quadric and degree $a$ generator of $C$ yield a minimal Koszul relation in $J$. \\

To summarize, Cases 1 and 2 are impossible, and Case 3 only happens when the quadric and degree $a$ generator of $C$ yield a minimal Koszul relation in $J$. Moreover, in this case, $a=2n+4$. 
\end{proof}


\begin{lem} \label{Koszul trim}
Fix $n\geq 4$ and let $J$ be an ideal satisfying $(\star)$. Further, let $C\subseteq J$ be a grade $3$ complete intersection of type $(2, 2n+4, b)$, where $b\geq 2n+6$, and such that the quadric and degree $2n+4$ generators of $C$ yield a minimal Koszul relation in $J$. Then any sequentially bounded double link of $J$ with respect to $(2,2n+4,b)$ satisfies $(\star)$ or $(\vardiamond)$. 
\end{lem}

\begin{proof}
The quadric and degree $2n+4$ generators of $C$ must form part of a minimal generating set of $J$ since they yield a minimal Koszul relation in $J$. Then by Lemma \ref{min Koszul relation}, the link $J_1=C:J$ has a free resolution $\mc{F}_{\bigcdot}'$ with terms
\begin{align} \notag
F_1' &= R(-2) \oplus R(-(b-n+1)) \oplus R(-(b-1)) \oplus R(-(b+2)) \oplus R(-(2n+4)) \\ \notag
F_2' &= R(-(b-n+2)) \oplus R(-b) \oplus R(-(b+3)) \\ \notag
&\oplus R(-(b+n+2)) \oplus R(-(b+n+4))^2 \oplus R(-(2n+6)) \\ \notag
F_3' &= R(-(b+n+3)) \oplus R(-(b+n+5))^2  
\end{align}
Note that $b-1\geq 2n+5$. Moreover, if $(J_1)_{\leq \max \{b-n+1, 2n+4\} }$ has grade $3$, then $J_1$ contains a regular sequence of degrees $2$, $2n+4$, and $b-n+1$. In particular, using the quadric and degree $2n+4$ generators of $C$, we may complete this regular sequence with a degree $b-n+1$ generator of $J_1$. Then by Lemma \ref{ci shift trims}, link of $J_1$ with respect to this regular sequence yields a contradiction with Lemma \ref{max degree relations}. Thus $(2, 14, b-1)$ is the minimal type in $J_1$. 
\par If $J_2$ is a minimal link of $J_1$, then using Ferrand's resolution of $J_2$, it is easily checked that the only difference in the Betti table of $J_2$ and the Betti table of an ideal satisfying $(\vardiamond)$ is the value $\beta_{i,b-1}(R/J_2)$ for $i=1,2$. However, this graded Betti number is $0$ by Lemma \ref{max degree relations}. The same argument applies to show that any link of $J_1$ with respect to $(2, 2n+4, b)$ satisfies $(\star)$. 
\end{proof}



\begin{prop} \label{with quadratic}
Fix $n\geq 4$ and let $J$ be an ideal satisfying $(\star)$. Then any sequentially bounded double link of $J$ with respect to $(2,a,b)$, with $b\geq 2n+6$, satisfies $(\star)$ or $(\vardiamond)$. 
\end{prop}

\begin{proof}
Let $C\subseteq J$ be a complete intersection of type $(2,a,b)$ with $b\geq 2n+6$, and set $J_1=C:J$. Since the only minimal Koszul relation in $J$ is between the quadric and degree $2n+4$ generator, we may assume by Lemmas \ref{deduce trim} and \ref{Koszul trim} that $(2,a,b)$ is the minimal type in $J_1$. Let $C_1\subseteq J_1$ be a complete intersection with type $(2,a,b)$, and set $J_2=C_1:J_1$. We show $J_2$ satisfies $(\star)$. Since $J_2$ is a ghost double link of $J$, 
\[ \beta_{i,j}(R/J)\leq \beta_{i,j}(R/J_2) \leq  \beta_{i,j}(R/J)+1 \text{ for } 1\leq i\leq 2 \text{ and } j\in\{a,b\},\] 
and $\beta_{i,j}(R/J_2)=\beta_{i,j}(R/J)$ otherwise. However, since $b\geq 2n+6$, Lemma \ref{max degree relations} implies $\beta_{i,b}(R/J_2)=\beta_{i,b}(R/J)=0$ for $i=1,2$. To show the same is true for $a$, By Lemma \ref{ci shift trims}, we may assume the degree $a$ generator of $C$ is not minimal in $J$, and the degree $b$ generator of $C_1$ is not a weak associate of the degree $b$ generator of $C$ in $J_1$. Since $g_3(J_1)=b$, this is only possible if $b\in \{a+b-2n-2, a+b-2n-5, a+b-3n-3\}$, which implies $a\in \{2n+2, 2n+5, 3n+3\}$. If $a=3n+3$, then again Lemma \ref{max degree relations} implies $\beta_{i,a}(R/J_2)=\beta_{i,a}(R/J)$ for $i=1,2$. Thus we may assume either $a=2n+4$ or $a=2n+2$. 
\par Assume $a=2n+4$ and $\beta_{i,2n+4}(R/J_2) = \beta_{i,2n+4}(R/J)+1$ for $i=1,2$. Then $\beta_{1,2n+4}(R/J_2)=2=\beta_{2,2n+6}(R/J_2)$, and since $\beta_{2,2n+5}(R/J_2)=0$, it follows that for any choice of a minimal generating set of $J_2$, the quadric yields a minimal Koszul relation in $J_2$ with both minimal generators of degree $2n+4$. In particular, there is a complete intersection $C_2\subseteq J_2$ with type $(2,2n+4,2n+4)$ such that both minimal Koszul relations in $J_2$ appear from the generators of $C_2$. Using Ferrand's resolution of $J_3=C_2:J_2$ and Proposition \ref{min Koszul relation}.ii, we have $\beta_{1,2n+3}(R/J_3)=1$ and $\beta_{2,j}(R/J_3)=0$ for $2n+3\leq j\leq 2n+5$, which is impossible by the Koszul argument (recall Definition \ref{Koszul argument}). Therefore, $\beta_{i,2n+4}(R/J_2) = \beta_{i,2n+4}(R/J)$ for $i=1,2$, so $J_2$ satisfies $(\star)$.
\par Assume $a=2n+2$ and $\beta_{i,2n+2}(R/J_2) = \beta_{i,2n+2}(R/J)+1=1$ for $i=1,2$. Then $g_2(J_2)=n+1$ by Lemma \ref{g1q1}.i, so there is a complete intersection $C_2\subseteq J_2$ with type $(2,n+1,2n+4)$. If $J_3=C_2:J_2$, the only difference between the Betti table of $J_3$ and the Betti table of an ideal with resolution (\ref{I res}) are the values
\[ \beta_{i,n+1}(R/J_3)\leq 1 \text{  for  } 1\leq i\leq 2\]
and 
\[ \beta_{i,n+5}(R/J_3)\leq 1 \text{  for  } 2\leq i\leq 3.\] 
The values of these graded Betti numbers for an ideal with resolution (\ref{I res}) are all $0$. However, $\beta_{2,j}(R/J_3)=0$ for $n+2\leq j\leq n+3$, so $\beta_{i,n+1}(R/J_3)=0$ for $1\leq i\leq 2$ by the Koszul argument. Now by Lemma \ref{g12l}, the sub-ideal $(J_3)_{\leq n+1}$ is a grade $2$ perfect ideal. Therefore, if $\beta_{3,n+5}(R/J_3)=1$, the degree $n+3$ minimal generator of $J_3$ is a linear multiple of a socle generator for $J_3$. But $\beta_{2,n+4}(R/J_3)=1$, which contradicts Lemma \ref{relations from socle}. Hence $\beta_{i,n+5}(R/J_3)=0$ for $2\leq i\leq 3$, which by Proposition \ref{min Koszul relation} is only possible if $\beta_{i,2n+2}(R/J_2)=0$ for $1\leq i\leq 2$. Therefore, $J_2$ satisfies $(\star)$.

\end{proof}

\section{Monomial Ideals with the Same Hilbert Function}

In this final section we consider the Artinian case and use Macaulay's inverse system and linkage to find for every ideal satisfying $(\star)$ a licci monomial ideal with the same Hilbert function. This provides further evidence for a positive answer to Question \ref{Monomial HF Q}. We refer the reader to \cite{GHMS07} (Chapter 5) for further background in Macaulay's inverse systems.  

\par Let $R=k[x_1,\dots, x_N]$ with homogeneous maximal ideal $\mf{m}$ as before, and let $T=k[X_1,\dots, X_N]$. There is a \textit{contraction} operation which, on the monomials of $R$ and $T$, takes the form $$x_1^{a_1}\dots x_N^{a_N} \circ X^{b_1}\dots X_N^{b_N} = X_1^{b_1-a_1}\dots X_{N}^{b_N-a_N},$$ where $X_i^{a_i}=0$ if $a_i<0$. This operation induces a one-to-one correspondence
\[ \{ R\text{-ideals} \} \leftrightarrow \{ R\text{-submodules of } T\} \]
where an $R$-ideal $I$ corresponds to the $R$-submodule 
\[ I^{-1} := \{ G\in T: f\circ G=0 \text{ for all } f\in I\}. \]
Conversely, given an $R$-submodule $M$ of $T$, if $I=\text{ann}_R(M)$, then $I^{-1}=M$. If $\{ m_{\alpha} \}$ is a generating set of $M$, then we write $M=\left< m_{\alpha} \right>_R$ to denote that the action on $M$ is from $R$. 

\begin{prop}[\cite{GHMS07}] \label{HF from dual}
With the set-up above, let $I$ be a homogeneous $R$-ideal. 
\begin{enumerate}[i.]
\item The inverse system $I^{-1}$ is a finitely generated $R$-submodule of $T$ if and only if $I$ is $\mf{m}$-primary, in which case the size of a minimal generating set of $I^{-1}$ is the CM type of $I$. 
\item For all $j$,
\[ \dim_k (I^{-1})_j = \dim_k (R/I)_j =: \text{HF}_{R/I}(j) .\]
\end{enumerate}
\end{prop}

The following lemma shows an interesting connection between Macaulay's inverse system and linkage.

\begin{lem} \label{dual generators from link}
Let $I$ be an $\mf{m}$-primary $R$-ideal generated by $\{f_1,\dots, f_n\}$ and let $C\subseteq I$ be a Gorenstein ideal with $C^{-1} = \left< G \right>_R$. Then 
\[ (C:I)^{-1} = \left< f_1\circ G,\dots, f_n\circ G \right>_R.\]
\end{lem}

\begin{proof}
The containment $f\in C:I$ holds if and only if $ff_i\in C$ for all $i=1,\dots, n$. This is equivalent to $ff_i\circ G=0$ for all $i=1,\dots, n$, which is true if and only if $f\circ (f_i \circ G)=0$ for all $i=1,\dots, n$, which is equivalent to $f\in \left< f_1\circ G, \dots, f_n\circ G \right>_R$.
\end{proof}

The last proposition we need is the formula mentioned in the introduction which displays a relationship between the Hilbert function of directly linked $\mf{m}$-primary homogeneous ideals. 

\begin{prop}[\cite{M12}, 5.2.19] \label{HF links}
Let $I$ be an $\mf{m}$-primary $R$-ideal containing a complete intersection $C$ of type $(a_1,\dots, a_N)$. Set $a=a_1+\dots +a_N - N$. Then 
\[ \text{HF}_{R/(C:I)}(a-i) = \text{HF}_{R/C}(i) - \text{HF}_{R/I}(i) \]
for all $i$.
\end{prop}

\begin{thm}
\label{thm:main2}
Let $R=k[x,y,z]$ and fix $n\geq 4$. Let $J$ be an ideal satisfying $(\star)$. Consider the monomial ideal $$J'=(x^2, y^{2n+1}, z^{2n+4}, y^nz, xy^{n-1}z, xyz^{n+1}).$$ 
Then $J'$ is SBL and has the same Hilbert function as $J$.  
\end{thm}

\begin{proof}
First, using the construction of Theorem \ref{HMNU}, let $$I=z(z^{2n+3}, x,y)+(x^n, y^{n+3}).$$ Then the link $$(yz,x^{n+1}, y^{2n+4}-z^{2n+4}):I$$ realizes the Hilbert function of an ideal satisfying $(\star)$. Indeed, we know that the graded resolution of $I$ has the form of (\ref{I res}), which implies the graded resolution of the link has the form of (\ref{J res}). Therefore, we may assume $J$ is the link of $I$ above without the loss of generality. 
\par A sequence of links from $J'$ to a complete intersection is as follows:
\[ J' \widesim{(x^2,y^{2n+1},z^{2n+4})} J_1' \widesim{(x^2,y^{2n},z^{2n+3})} J_2' \widesim{(x^2,y^{n-1},z^{2n+3})} J_3' \widesim{(x,y^{n-1},z^{2n+3})} (x,y^{n-2},z^n) .\]
Therefore, $J'$ is SBL. To show $J$ and $J'$ have the same Hilbert function, we use Proposition \ref{HF links} and Macaulay's inverse systems.
\par Let $C'=(x^2, y^nz, y^{2n+4}-z^{2n+4})$ and set $I' = C':J'$. Since
\[ I \widesim{(2,n+1,2n+4)} J \hspace{7pt} \text{and} \hspace{7pt} I' \widesim{(2,n+1,2n+4)} J',\]
by Proposition \ref{HF links} it suffices to show $I$ and $I'$ have the same Hilbert function. 
\par It is easily checked that $I^{-1} = \left< X^{n-1}Y^{n+2}, Z^{2n+3} \right>_R$. Let $$G=XY^{3n+3} + XY^{n-1}Z^{2n+4} \in T.$$ A simple calculation shows $C' \subseteq 0:_R G$, and the Hilbert functions of $C'$ and $0:_R G$ are the same since both are symmetric with highest non-vanishing degree $2n+4$ (see \cite{GHMS07}, Chapter 5). Thus $(C')^{-1} = \left< G \right>_R$. By Lemma \ref{dual generators from link},
\begin{align} \notag
 (I')^{-1} &= (C' : J')^{-1} = (C':(xy^{n-1}z, xyz^{n+1}, y^{2n+1}))^{-1} \\ \notag
 &= \left< xy^{n-1}z \circ G, xyz^{n+1} \circ G, y^{2n+1} \circ G \right>_R \\ \notag
 &= \left< Z^{2n+3}, Y^{n-2}Z^{n+3}, XY^{n+2} \right>_R. \notag
 \end{align}
Notice that $I^{-1}\cap (I')^{-1}= \left< XY^{n+2}, Z^{2n+3} \right>_R$. Then for all $j$, 
\begin{align} \notag
\dim_k(I^{-1} \bs (I^{-1}\cap (I')^{-1}))_j &= \#\{ X^aY^b : a+b=j, \hspace{5pt} 2\leq a\leq n-1, \text{ and } 0\leq b\leq n+2 \}  \\ \notag
&= \#\{ X^aY^b : a+b = j, \hspace{5pt} 1\leq a \leq n-2, \text{ and } 1\leq b \leq n+3\} \\ \notag
&= \#\{ Y^bZ^c : b+c = j, \hspace{5pt} 1\leq b \leq n-2, \text{ and } 1\leq c \leq n+3\} \\ \notag
&=\dim_k((I')^{-1} \bs (I^{-1}\cap (I')^{-1}))_j.
\end{align}
Therefore, $\dim_k(I^{-1})_j = \dim_k(I')^{-1}_j$ for all $j$, which by Proposition \ref{HF from dual} implies $I$ and $I'$ have the same Hilbert function.
\end{proof}


\appendix
\section{Grade Jumps and Linear Syzygies} 

Let $R=k[x_1,\dots, x_N]=\oplus_{i=0}^{\infty} R_i$, with the standard grading, and let $I$ be a grade $n$ homogeneous ideal. The lemmas of this section are mainly concerned with how to find the grade jumps of grade $3$ perfect ideal using only the graded Betti numbers. Most of these lemmas are really remarks, but since they are used often in the proof of the Main Theorem, we record them here in detail.

\begin{lem} \label{g1q1}
Let $q$ be a quadratic form, and let $f,g$ be nonzero, homogeneous polynomials of degrees $n$ and $m$, respectively, with $n\geq m>1$. Let $H=(q,f,g)$.
\begin{enumerate}[i.]
\item If $grade(H)=1$, then $\beta_{2,j}(R/H)\neq 0$ for $j\in\{n+1, m+1\}$, and $\{\beta_{2,j}(R/H): j< n+m\}$ has at least three nonzero elements. 
\item Assume $n=m$ and that $I'$ is an $R$-ideal with $\alpha :=  \min\{ j : \beta_{1,j}(R/I')\neq 0\} > n$. Consider the ideal $I=(H,I')$ and the set $S=\{j : j<2n \text{ and } \beta_{2,j}(R/I)\neq 0\}$. Let $p=\#S$. If $p<3$, or if $p= 3$ and either $\max \{j : j\in S\}=\alpha+1$ and $\beta_{2,\alpha+2}(R/I)=0$ or $\alpha+1\leq \max \{j: j\in S\}\leq \alpha+2< 2n$, then $\grade(H)\geq 2$.
\end{enumerate}
\end{lem}

\begin{proof}
To prove i., write $H = \ell H'$, which is possible because $H$ is a grade $1$ ideal containing a quadric. The result follows by the short exact sequence
\[ 0 \to R/ H'(-1) \xrightarrow{\cdot \ell} R/ H \to R/(\ell) \to 0.\]
Part ii. essentially follows by the Koszul argument. Indeed, if $p<3$, then $\grade(H)\geq 2$ by the previous part. Assume $p=3$. Then, since $\alpha>n$, we have $\beta_{2,n+1}(R/I)=2$. Moreover, $\beta_{2,j}(R/I)=1$ for exactly one $j$ with $\alpha+1 \leq j\leq \alpha+2$. With the conditions of the statement, the Koszul argument implies the corresponding syzygy must involve the quadric and one of the degree $\alpha$ generators of $I$, which by part i. implies $\grade(H)\geq 2$.
\end{proof}

\begin{lem} \label{max degree relations}
Let $I=(f_1, \dots, f_n)$ be homogeneous $R$-ideal with $\deg f_1 \leq \dots \leq \deg f_{n}$. Set $I'=(f_1,\dots, f_{n-1})$, and let $S=\{ j : j > \deg f_n \text{ and } \beta_{2,j}(R/I)\neq 0\}$. Then $$\grade(I') \leq \sum_{j\in S} \beta_{2,j}(R/I).$$
\end{lem}

\begin{proof}
Let $\mc{F}_{\bigcdot}$ be the minimal graded free resolution of $I$, and let $K=\text{Syz}_1(R/I)$. Let $\{e_1,\dots, e_n\}$ be a basis for $F_1$, and set $m = \sum_{j\in S} \beta_{2,j}(R/I)$. Then for $i=1,\dots, m$, there are syzygies $$s_i = g_{1i} e_1 + \dots + g_{ni} e_n \in K$$ which form a minimal generating set of $K/(K_{\leq \deg f_n})$. The Koszul relations $f_n e_i - f_i e_n\in K$, for $i=1,\dots, n-1$, can be written as $R$-combinations of the $s_i$ and syzygies of degree less than $\deg f_n$. Since the lower degree relations lie in the span of $\{e_1,\dots, e_{n-1}\}$ (indeed, recall how we define degree in Defintion \ref{syzygy defs}), it follows that $(f_1,\dots, f_{n-1}) \subseteq (g_{n1},\dots, g_{nm})$. Thus the result follows by the Principal Ideal Theorem.
\end{proof}

\begin{lem}[\cite{HMNU07}, Remark 3.1] \label{max socle shift}
Let $I$ be a perfect grade $n$ $R$-ideal with $M=\max \{j : \beta_{n,j}\neq 0\}$. If $I$ contains a complete intersection of type $(a_1, \dots, a_n)$, then $a_1+\dots+a_n\geq M$, with equality if and only if $I$ is a complete intersection.
\end{lem}

These last lemmas are concerned with special cases of linear relations.



\begin{lem} \label{g12l}
Assume $g\in R$ is a form of degree $n>0$ and $\ell$ is a linear form such that $g\notin (\ell)$. Let $H\subseteq R$ be a nonzero, proper ideal, and let $\mathcal{H}_{\bigcdot}$ be the minimal graded free resolution of $H$.
\par If the colon ideal $(\ell H):g$ contains a linear form, then $(\ell H):g=\ell R$, and the minimal graded free resolution of $(\ell H,g)$ is $\mathcal{H}_{\bigcdot}'$, where $H_1'=H_1(-1)\oplus R(-n)$, $H_2'=H_2(-1)\oplus R(-(n+1))$, and $H_i' = H_i(-1)$ for all $i>2$. In particular, if pd$(H)\geq 2$, then $\ell H$ and $(\ell H,g)$ have the same projective dimension.
\end{lem}

\begin{proof}
Since $g\notin (\ell)$, the sequence $g,\ell$ is a regular sequence. Thus for some linear form $L$, 
\[L R \subseteq (\ell H):g\subseteq \ell R : g = \ell R,\]
which implies $LR=\ell R$ and $(\ell H):g = \ell R$. The second and third statements are an immediate consequence of the short exact sequence
\[ 0 \to R/(\ell)(-n) \xrightarrow{\cdot g} R/\ell H \to R/(\ell H,g) \to 0 \]
and the Horseshoe Lemma.
\end{proof}

\begin{defn}
For an $\mf{m}$-primary ideal $I$, the module $(I:\mf{m})/I$ is called the \textit{socle} of $R/I =: \bar{R}$. If $\{ \bar{f_1}, \dots, \bar{f_r} \}$ minimally generate the socle of $R/I$, then we say that $f_1,\dots, f_r$ are \textit{socle generators} for $I$.
\end{defn}

\begin{lem} \label{relations from socle}
Let $I$ be an $\mf{m}$-primary ideal. If a minimal generator of degree $n$ is a linear multiple of a socle generator, then $\beta_{2,n+1}(R/I)\geq 2$.
\end{lem}

\begin{proof}
Suppose $f$ is a minimal generator of degree $n$ and let $s$ be a socle generator with $f=\ell s$ for some linear form $\ell$. Choose two linear forms $\ell_1,\ell_2$ for which $\ell,\ell_1$, and $\ell_2$ are linearly independent. Then $\ell (\ell_i s) = \ell_i (\ell s)$, for $i=1,2$, are distinct linear relations, which must be minimal by degree reasons and linear independence. 
\end{proof}


\section{Double Link Trims} In this section we handle several technical proofs which contribute to the proof of the Main Theorem. There are four statements, each of which are concerned with double links of ideals satisfying $(\vardiamond)$, $(\star)$, or $(\star \star)$ (recall Definitions \ref{min link res}, \ref{main ideals}, and \ref{star star}, respectively). Ghost double links of ideals satisfying $(\vardiamond)$ and $(\star)$ repeat the conditions, and double links of ideals satisfying $(\star \star)$ with a drop in the type satisfy $(\vardiamond)$. 


\begin{prop} \label{minimal link of I trim}
Fix $n\geq 4$ and assume $L$ satisfies $(\vardiamond)$. Then any sequentially bounded double link of $L$ with respect to $(2,c,d)$ satisfies $(\star)$. 
\end{prop}

\begin{proof}
Let $D\subseteq L$ be a complete intersection of type $(2,c,d)$ and set $L_1=D:L$. Recall from Proposition \ref{minimal link of I} that $c\geq n$, $d\geq 2n+4$, and $D$ achieves the minimal type in $L_1$. Let $D_1\subseteq L_1$ be a complete intersection of minimal type and set $L_2=D_1:L_1$. Then $L_2$ is a ghost double link of $L$ with
\[ \beta_{i,j}(R/L)\leq \beta_{i,j}(R/L_2) \leq  \beta_{i,j}(R/L)+1 \text{ for } 1\leq i\leq 2 \text{ and } j\in\{c,d\},\] 
\par To show $\beta_{i,d}(R/L_2)=\beta_{i,d}(R/L)$ for $i=1,2$, by Lemma \ref{ci shift trims}.i. we may assume the degree $d$ generator of $D$ is not minimal in $L$ and the generator of degree $c$ in $D_1$ is not a weak associate of the degree $c$ generator of $D$. In particular, since $g_3(L)=2n+4$, we may assume $d\geq 2n+5$. Moreover, since $g_2(L_1)=c$, we have $c=c+d-3n-2$ or $c=c+d-2n-1$. The latter cannot happen because $d>2n+1$, so $c=c+d-3n-2$, which means $d=3n+2$. But $\sum_{j>3n+2} \beta_{2,j}(R/L_2)=1$, so $\beta_{i,3n+2}(R/L_2)=0$ by Lemma \ref{max degree relations}. Thus $\beta_{i,d}(R/L_2)=\beta_{i,d}(R/L)$ for $i=1,2$.
\par Assume $\beta_{i,c}(R/L_2)\neq \beta_{i,c}(R/L)$ for $i=1,2$. Then repeating the argument in the previous paragraph we see $c+d-3n-2=d$ or $c+d-2n-1 = d$, so $c=2n+1$ or $c=3n+2$. The latter case yields a contradiction with Lemma \ref{max degree relations}. If $c=2n+1$, then $L_2$ contains a complete intersection $C_2$ with type $(2,n,2n+4)$. The link $L_3=C_2:L_2$ has a free resolution $\mc{P}_{\bigcdot}''$ with terms
\begin{align} \notag
P_1'' &= R(-2)^2 \oplus R(-n) \oplus R(-(n+3)) \oplus R(-(2n+4)) \\ \notag
P_2'' &= R(-3) \oplus R(-(n+1)) \oplus R(-(n+4)) \oplus R(-(n+5)) \\ \notag
&\oplus R(-(2n+3)) \oplus R(-(2n+5))^2 \\ \notag
P_3'' &= R(-(n+5)) \oplus R(-(2n+4)) \oplus R(-(2n+6)) 
\end{align}
Since $\beta_{i,2n+1}(R/L_2)=1$ for $i=1$, we have $\beta_{3,n+5}(R/L_3)=1$, so there is a socle generator $p$ for $L_3$ with degree $n+2$. However, by Lemma \ref{g12l}ii., the sub-ideal $(L_3)_{\leq n}$ is a grade $2$ perfect ideal. Thus there is a linear form $\ell$ regular on $(L_3)_{\leq n}$, and since $\ell p\notin (L_3)_{\leq n}$, we may assume that $\ell p$ is part of a minimal generating set of $L$. This is impossibe by Lemma \ref{relations from socle}. Thus $\beta_{i,c}(R/L_2)\neq \beta_{i,c}(R/L)$ for $i=1,2$ trims, so $L_2$ satisfies $(\vardiamond)$.
\end{proof}



\begin{prop} \label{H a drop}
Fix $n\geq 4$ and assume $H$ satisfies $(\star \star)$. If there is a sequence of links $H\widesim{(2,a,b)} H_1 \widesim{(2,a-1,b)} H_2$, then $H_2$ satisfies $(\vardiamond)$.
\end{prop}

\begin{proof}
Assume $H_1=C:H$ with $\text{type}(C)=(2,a,b)$. The proof of Lemma \ref{ghost g2} shows that such a sequence of links implies that $b=2n+4$ and $(2,a-1,b)$ is the minimal type of $H_1$. The graded free resolution of $H_1$ is described by resolution (\ref{H1 res}) in the proof of Proposition \ref{ghost summands}. It follows that the graded free resolution of $H_2$ is the same as resolution (\ref{L res}) from which $(\vardiamond)$ is defined, except possibly with $R(-(a-1))\oplus R(-(2n+1))^{s} \oplus R(-(2n+4))^t$ in steps $1$ and $2$ of the resolution. We claim this trims. 
\par By a similar argument to the proof of Lemma \ref{ci shift trims}, we may assume the degree $a$ generator of $C$ is not a minimal generator of $H$, and that the degree $b$ form used to define $H_2$ is not a weak associate of the degree $b$ generator of $C$. Since $b=2n+4$ and $g_3(H_1)=b$, we have $2n+4\in \{ a-n+1,a+2\}$, so $a\in \{3n+3,2n+2\}$. But $a\leq 2n+4 < 3n+3$, so $a=2n+2$. By Lemma \ref{g1q1}, $H_2$ contains a complete intersection $C_2$ with type $(2,n,2n+4)$. The link $H_3=C_2:H_2$ has a free resolution $\mc{G}_{\bigcdot}''$ with terms
\begin{align} \notag
G_1'' &= R(-2)^2 \oplus R(-n) \oplus R(-(n+3)) \oplus R(-(2n+4)) \\ \notag
G_2'' &= R(-3) \oplus R(-(n+1)) \oplus R(-(n+4)) \oplus \\ \notag
&R(-(2n+3)) \oplus R(-(2n+5))^2 \oplus R(-(n+2))^t \oplus R(-(n+5))^{s+1} \\ \notag
G_3'' &= R(-(2n+4)) \oplus R(-(2n+6)) \oplus R(-(n+2))^t \oplus R(-(n+5))^{s+1}.
\end{align}
By Lemma \ref{g12l}, the sub-ideal $H':=(H_3)_\leq n$ has grade $2$ and is perfect. If $t>0$, then $\beta_{3,n+2}(R/H_3)>0$, which implies $\beta_{3,n+2}(R/H')>0$ (since $(H')_{\leq n} = H'$). Thus $t=0$. Moreover, if $\beta_{3,n+5}(R/H_3)>0$, then the degree $n+3$ minimal generator of $H_3$ must be a linear multiple of a socle generator for $H_3$. But $\beta_{2,n+4}(R/H_3)=1$, which contradicts Lemma \ref{relations from socle}. Therefore, $s+1=0$, and tracing the units back to the resolution of $H_2$, we see that $H_2$ satisfies $(\vardiamond)$.
\end{proof}


\begin{prop} \label{H b drop}
Fix $n\geq 4$ and assume $H$ satisfies $(\star \star)$. If there is a sequence of links $H\widesim{(2,a,b)} H_1 \widesim{(2,a,b-1)} H_2$, then $H_2$ satisfies $(\vardiamond)$.
\end{prop}

\begin{proof}
Using the proof of Lemma \ref{ghost g3}, the sequence of links implies $a=2n+4$, $b=2n+5$. Substituting these values into the terms of $(\ref{H2 res})$, we see that $H_2$ has a free resolution $\mc{G}_{\bigcdot}''$ with terms
\begin{align} \notag
G_1'' &= R(-2) \oplus R(-n)^2 \oplus R(-(n+2)) \\ \notag
&\oplus R(-(2n+1))^s \oplus R(-(2n+3)) \oplus R(-(2n+4))^{t+2} \\ \notag
G_2'' &= R(-(n+1))^2 \oplus R(-(n+3)) \oplus R(-(2n+1))^s \oplus R(-(2n+2)) \\ \notag
&\oplus R(-(2n+3)) \oplus R(-(2n+4))^{t+1} \oplus R(-(2n+5))^{2} \oplus R(-(3n+3)) \\ \notag
G_3'' &= R(-(2n+3)) \oplus R(-(2n+5)) \oplus R(-(3n+4)).
\end{align}
We show $\beta_{i,2n+5}(R/H_2)=0$ for $i=2,3$, that $\beta_{i, 2n+1}(R/H_2)=\beta_{i, 2n+3}(R/H_2)= 0$ for $i=1,2$, and that $\beta_{1,2n+4}(R/H_2)=1$ and $\beta_{2,2n+4}(R/H_2)=0$.
\par By Lemma \ref{g1q1}, we have $g_2(H_2)=n$. So let $H_3$ be the link of $H_2$ with respect to $(2,n,2n+4)$. Then $H_3$ has a free resolution $\mc{G}_{\bigcdot}'''$ with terms
\begin{align} \notag
G_1''' &= R(-2)^2 \oplus R(-n) \oplus R(-(n+1)) \oplus R(-(n+3)) \oplus R(-(2n+4)) \\ \label{G3 ha}
G_2''' &= R(-3) \oplus R(-(n+1))^2 \oplus R(-(n+4)) \oplus R(-(2n+3)) \\ \notag
&\oplus R(-(2n+5))^2 \oplus R(-(n+2))^{t+1} \oplus R(-(n+3)) \oplus R(-(n+5))^s \\ \notag  
G_3''' &= R(-(n+2))^{t+1} \oplus R(-(n+3)) \oplus R(-(n+5))^s \\ \notag
&\oplus R(-(2n+4)) \oplus R(-(2n+6)) 
\end{align}
Lemma \ref{g12l} implies $R(-(n+1))$ in $G_1'''$ and $G_2'''$ trim, and that $(H_3)_{\leq n}$ is a perfect grade $2$ ideal. Thus there cannot be socle generators for $H_3$ of degrees $n-1$ or $n$, which implies $R(-(n+2))^{t+1}\oplus R(-(n+3))$ in $G_2'''$ and $G_3''$ trim. Lastly, as we have seen before, the fact that $(G_3)_{\leq n}$ is a grade $2$ perfect ideal (by Lemma \ref{g12l}) and $\beta_{2,n+4}(R/H_3)=1$ implies $\beta_{3,n+5}(R/H_3)=0$. Tracing back the unit entries to (\ref{G3 ha}), we have all of the desired trims to conclude $H_2$ satisfies $(\vardiamond)$.  
\end{proof}


\begin{prop} \label{big a1 trim}
Fix $n\geq 4$ and assume $J$ satisfies $(\star)$. Then any sequentially bounded double link of $J$ with respect to $(a_1,a_2,a_3)$, where $a_1>2$, satisfies $(\star)$.
\end{prop}

\begin{proof}
Let $C\subseteq J$ be a complete intersection with type $(a_1,a_2,a_3)$, where $a_1>2$. By Proposition \ref{without quadratic}, we know that $(a_1,a_2,a_3)$ is the minimal type in $J_1=C:J$. If $C_1\subseteq J_1$ is a complete intersection with type $(a_1,a_2,a_3)$ and $J_2=C_1:J_1$, then $J_2$ is a ghost double link of $J$. Thus to prove $J_2$ satisfies $(\star)$, we need to show $\beta_{i,a_j}(R/J_2)=\beta_{i,a_j}(R/J)$ for $1\leq i\leq 2$ and $1\leq j\leq 3$. By Lemma \ref{ci shift trims}, we may assume that for $\{i,j,k\}=\{1,2,3\}$, the degree $a_i$ generator of $C$ is not minimal in $J$, and that the degree $a_j$ and $a_k$ generators of $C_1$ are not both weak associates of the degree $a_j$ and $a_k$ generators of $C$, respectively. \\
 
\noindent \textbf{Step 1:} Show $\beta_{i,a_1}(R/J_2)=\beta_{i,a_1}(R/J)$ for $i=1,2$. \\ 

We consider three cases: when $a_1\leq n+1$, when $a_1=n+2$, and when $a_1\geq n+3$. 
\par If $a_1\leq n+1$, then since $\sum_{j\leq n} \beta_{1,j}(R/J_2)=1$, we have $\text{Syz}_1(R/J_2)_{\leq n+1}=0$. Thus $\beta_{i,a_1}(R/J_2)=0=\beta_{i,a_1}(R/J)$ for $1\leq i\leq 2$. 
\par Assume $a_1=n+2$ and that $\beta_{i,a_1}(R/J_2)=1$ for $1\leq i\leq 2$. Write $(J_2)_{\leq n+1} = (q,f_1,f_2)$. If $q,f_1$ is a regular sequence, then using the mapping cone induced by 
\[ 0\to R/(q,f_1):f_2(-(n+1)) \xrightarrow{ \cdot f_2} R/(q,f_1)\to R/(q,f_1,f_2)\to 0\]
to resolve $R/(q,f_1,f_2)$, it follows that $(q,f_1):f_2$ contains three linearly independent linear forms, which is impossible since one of the linear forms is regular on $R/(q,f_1)$ and $f_2\notin (q,f_1)$. By the same reasoning, neither $q,f_2$ or $f_1,f_2$ can be a regular sequence. In particular, $g_2(J_2)>n+1$, so we can write $(q,f_1,f_2)=\ell(\ell',g_1,g_2)$ for some linear forms $\ell$ and $\ell'$ and degree $n$ forms $g_1$ and $g_2$. Further, using the short exact sequence
\[ 0\to R/(\ell',g_1,g_2) \xrightarrow{\cdot \ell} R/(q,f_1,f_2) \to R/(\ell)\to 0,\]
we see that there are linear forms $L$ and $L'$ and a degree $n-1$ form $h$ for which $g_1=Lh$ and $g_2=L'h$. Now $\text{Syz}_1(R/J_2)_{\leq n+2}$ can be generated by syzygies $s_1,s_2,s_3$ corresponding to the relations
\[ g_1q=\ell'f_1, \hspace{10pt} g_2q=\ell'f_2, \hspace{5pt} \text{and} \hspace{5pt} L'g_1=Lg_2,\]
respectively. It is easily checked that 
\[ L' s_1 - L s_2 + \ell' s_3 =0,\]
which means that $\beta_{3,n+3}(R/(q,f_1,f_2))>0$. This would imply $\beta_{3,n+3}(R/J_2)>0$, which is not true, so we have a contradiction. Thus $\beta_{i,a_1}(R/J_2)=0$ for $1\leq i\leq 2$ when $a_1=n+2$.

\par Lastly, assume $a_1\geq n+3$ and that $\beta_{i,a_1}(R/J_2)>\beta_{i,a_1}(R/J)$ for $i=1,2$. Recall that $(\ref{J1 res no quad})$ gives the graded betti numbers of $J_1$. Since $a_1+a_3\geq 3n+7$, we have $a_2< a-3n-5$, which means the degree $a_2$ generator of $C_1$ must be a weak associate of the degree $a_2$ generator of $C$. Thus by Lemma \ref{ci shift trims}.ii, we know $\beta_{i,a_3}(R/J_2)=\beta_{i,a_3}(R/J)$. However, $g_3(J_2)=a_3$, so $a_3\in \{a-3n-5, a-2n-7, a-2n-4\}$. Since $a_1+a_2\geq 2a_1\geq 2n+6$, either $a_3=a-3n-5$ or $a_3=a-2n-7$.

\noindent (i.) $a_3=a-3n-5$. Then $a_1+a_2=3n+5$, and we may write $a_1=n+c$, where $3\leq c\leq n$. Thus $a_3>2n+3>a_1+2$, which by the Koszul argument implies $a_1+1\leq a_2 \leq a_2+2$ and $\beta_{i,a_2}(R/J_2)>\beta_{i,a_2}(R/J)$ for $i=1,2$. This implies $2c-4\leq n\leq 2c-3$. If $a_2+1\leq 2n+4\leq a_2+2$, then by the previous inequation, we have $n\leq 3$, which is impossible. Thus by the Koszul argument, $a_2+1\leq 2n+3\leq a_2+2$, which implies $4\leq n\leq 5$.  
\begin{itemize}
\item $n=4$. It follows that $c=4$, $a_1=8$, and $a_2=9$.
Then $\beta_{2,10}(R/J_2)=0$, and for any complete intersection $C_2\subseteq J_2$ with type $(2,9,12)$, the quadric yields a minimal Koszul relation with the degree $9$ and degree $12$ forms. If $J_3=C_2:J_2$, then $\beta{2,j}(R/J_3)=0$ for $11\leq j\leq 13$, but $\beta_{1,11}(R/J_3)=0$, which is a contradiction. 

\item $n=5$. It follows that $c=4$, $a_1=9$, and $a_2=11$.
Similar to the previous case, there is a complete intersection $C_2\subseteq J_2$ with type $(2,9,14)$ such that the quadric yields a minimal Koszul relation with the degree $9$ and degree $14$ forms. Then $J_3$ is an almost complete intersection containing a complete intersection $C_3$ with type $(2,8,11)$. Set $J_4=C_3:J_3$. The $\beta_{3,j}(R/J_4)=0$ unless $j=16$. Further $\beta_{2,3}(R/J_4)=2$ and $\beta_{1,2}(R/J_4)=3$. Thus the quadrics form a grade $2$ ideal, implying $J_4$ contains a complete intersection of type $(2,2,11)$, which contradicts Lemma \ref{max socle shift}.
\end{itemize}

\noindent (ii.) $a_3=a-2n-7$. Then $a_1+a_2=2n+7$, so $a_1=n+3$ and $a_2=n+4$. Then $\beta_{2,j}(R/J_2)=0$ for $n+5\leq j\leq n+6$, so by the Koszul argument we have $\beta_{1,n+4}(R/J_2)=0$, leaving $\beta_{2,n+4}(R/J_2)=1$. Further, $\beta_{1,n+3}(R/J_2)=2$ and $$p:=\sum_{n+4\leq j\leq n+5}\beta_{2,j}(R/J_2)=1,$$ 
but the Koszul argument applied to the quadric and both degree $n+3$ forms would imply $p\geq 2$. Thus we have a contradiction. \\


\noindent \textbf{Step 2:} Show $\beta_{i,a_3}(R/J_2)=\beta_{i,a_3}(R/J)$ for $i=1,2$. \\

Recall $g_3(J)=2n+4$, so we may assume $a_3\geq 2n+5$. Further, $a_2+a_3\geq 3n+6$, so $a-3n-5>a_1$. Thus the degree $a_2$ generator in $C_1$ is not a weak associate in $J_1$ of the degree $a_2$ generator of $C$. Since $g_2(J_1)=a_2$, in (\ref{J1 res no quad}) we see that $a_2 \in \{a-3n-5,a-2n-7,a-2n-4\}$. But since $a_1+a_3\geq 2n+8$, it is only possible that $a_2 = a-3n-5$, in which case $3\leq a_1\leq n$. Thus the degree $a_1$ generator of $C$ is not a minimal generator of $J$, but we know there was a trim of $R(-a_1)$ in the induced resolution of $J_2$. However, this trim cannot be from the minimal Koszul relation in $J_1$ between a weak associate of the degree $a_2$ and $a_3$ generators of $C$. Hence there must be another minimal Koszul syzygy of degree $a_2+a_3$ in $J_1$, which by (\ref{J1 res no quad}) implies $a_2+a_3\in \{ a-3n-4, a-2n-6, a-2n-3, a-n-4, a-n-2\}$. Then $a_1\in \{3n+4, 2n+6, 2n+3, n+4,n+2\}$, which is a contradiction. Therefore, the degree $a_2$ generator of $C_1$ is a weak associate of the degree $a_2$ generator of $C$, which by \ref{ci shift trims} concludes Step 2. \\
 

\noindent \textbf{Step 3:} Show $\beta_{i,a_2}(R/J_2)=\beta_{i,a_2}(R/J)$ for $i=1,2$. \\

Suppose $\beta_{i,a_2}(R/J_2)>\beta_{i,a_2}(R/J)$ for $1\leq i\leq 2$. Recall $g_2(J)=n+1$, so we may assume $a_2\geq n+2$. Then $a_2+a+3\geq 3n+6$, so $a-3n-5> a_1$. Further, since $g_3(J_1)=a_3$ and we assumed that the degree $a_3$ generator of $C_1$ is not a weak associate of the degree $a_3$ generator of $C$, by (\ref{J1 res no quad}) we see $a_3\in \{a-3n-5, a-2n-7, a-2n-4\}$, which implies $a_1+a_2\in \{3n+5, 2n+7, 2n+4\}$. In particular, $n+2\leq a_2\leq 3n+2$. To finish Step 3 we examine a particular partition of this interval of $a_2$ and show each part yields a contradiction.
\begin{itemize}

\item $a_2=n+2$: Since $\sum_{n+3\leq j\leq n+5}\beta_{2,j}(R/J_2)=1$ and $\sum_{n+2\leq j\leq n+3} \beta_{1,j}(R/J_2)=2$, the Koszul argument applied to the quadric and degree $n+3$ generator of $J_2$ implies $\beta_{i,n+2}(R/J_2)=0$, which contradicts our starting assumption. 

\item $n+3\leq a_2\leq 2n$: The Koszul argument proves this case.  

\item $a_2=2n+1$: The quadric and degree $a_2$ form yield a minimal Koszul relation in $J_2$. If $J_3$ is a link of $J_2$ with respect to $(2,a_2,2n+4)$ that uses the minimal Koszul relation, then applying the Koszul argument to the quadric and degree $2n+3$ minimal generator of $J_3$ yields a contradiction. 

\item $2n+2\leq a_2\leq 2n+4$: Let $J_3$ be a link of $J_2$ with respect to $(2, n+1, 2n+4)$ (such complete intersection exists by Lemma \ref{g1q1}). Notice that the generators of degree $2$ and $2n+4$ yield a minimal Koszul relation in $J_2$. Then $J_3$ has a free resolution $\mc{F}_{\bigcdot}''$ with graded free summands
\begin{align} \notag
F_1'' &= R(-2)^2 \oplus R(-n) \oplus R(-(n+3)) \oplus R(-(2n+4)) \\ \notag
F_2'' &= R(-3) \oplus R(-(n+1)) \oplus R(-(n+4)) \oplus R(-(2n+3)) \\ \notag
&\oplus R(-(2n+5))^2 \oplus R(-(3n+7-a_2)) \\ \notag
F_3'' &= R(-(3n+7-a_2)) \oplus R(-(2n+4)) \oplus R(-(2n+6)) 
\end{align}
Since in this case $n\leq 3n+7-a_2-3\leq n+2$, we have $\beta_{3,j}(R/J_3)\neq 0$ for some $j$ in the interval $n\leq j\leq n+2$. But Lemma \ref{g12l} implies the sub-ideal $(J_3)_{\leq n}$ is a grade $2$ perfect ideal, so $j=n+2$. Then the degree $n+3$ minimal generator of $J_3$ is a linear multiple of a socle generator, yielding a contradiction with Lemma \ref{relations from socle}. 

\item $a_2=2n+5$: Recall $a_1>2$ and $a_1+a_2\in \{3n+5,2n+7,2n+4\}$. Thus the only possibility is that $a_1+a_2 = 3n+5$, in which case $a_1 = n$. Since $\beta_{1,n}(R/J)=0$, the degree $a_1$ generator of $C$ lies in $\mf{m}J$. However, we know that $R(-a_1)$ trimmed in Ferrand's induced resolution of $J_2$ by Step 1, and we have assumed that the degree $a_3$ generator of $C_1$ is not a weak associate of the degree $a_3$ generator of $C$. Therefore, there must be another minimal Koszul relation in $J_2$ of degree $a_2+a_3$. In particular, $a_2+a_3\in \{a-3n-4, a-2n-6, a-2n-3, a-n-4, a-n-2, a_1+a_3\}$, which implies $a_1=a_2$ or $n=a_1\in \{n+2, n+4, 2n+3, 2n+6, 3n+4\}$. Thus $a_1=a_2$, which implies $n=2n+5$, which is a contradiction. 

\item $2n+6\leq a_2\leq 3n+2$: Lemma \ref{max degree relations} proves this case. 
\end{itemize}

We have now exhausted all possibilities for Step 3, so we may conclude that $\beta_{i,a_j}(R/J_2)=\beta_{i,a_j}(R/J)$ for $1\leq i\leq 2$ and $1\leq j\leq 3$. Therefore, $J_2$ satisfies $(\star)$.

\end{proof}



\end{document}